\documentclass[11pt]{article}

\usepackage[a4paper,margin=2.5cm]{geometry}
\usepackage{amsmath,amssymb,amsthm,amsfonts,amstext,bm,mathtools,amsopn}
\usepackage{booktabs,enumitem,microtype,xcolor}
\usepackage{multicol}
\usepackage{algorithm}
\usepackage{algpseudocode}
\usepackage{boxedminipage}
\usepackage{fancybox}
\usepackage{cases,supertabular}
\usepackage{multirow,makecell}
\usepackage{graphicx}
\usepackage{array}
\usepackage{listings}
\usepackage{subfigure}
\usepackage[sort&compress]{natbib}
\usepackage[colorlinks=true,allcolors=blue!55!black]{hyperref}
\usepackage{etoolbox}

\bibpunct[, ]{[}{]}{,}{n}{}{,}

\def\newblock{\ }

\DeclareGraphicsExtensions{.eps,.pdf,.png,.jpg}
\ifdefined\DeclareUnicodeCharacter
\DeclareUnicodeCharacter{2000}{ }
\DeclareUnicodeCharacter{2011}{\mbox{-}}
\DeclareUnicodeCharacter{2013}{--}
\fi

\theoremstyle{plain}
\newtheorem{theorem}{Theorem}
\newtheorem{lemma}{Lemma}
\newtheorem{assumption}{Assumption}
\theoremstyle{remark}
\newtheorem{remark}{Remark}

\newcommand{\compactalignspacing}{%
	\setlength{\abovedisplayskip}{0.15em plus 0.05em minus 0.03em}%
	\setlength{\belowdisplayskip}{0.15em plus 0.05em minus 0.03em}%
	\setlength{\abovedisplayshortskip}{0pt plus 0.03em}%
	\setlength{\belowdisplayshortskip}{0.12em plus 0.03em minus 0.03em}}
\AtBeginEnvironment{align}{\compactalignspacing}
\AtBeginEnvironment{align*}{\compactalignspacing}

\DeclareMathOperator{\dist}{dist}
\DeclareMathOperator*{\argmin}{arg\,min}
\allowdisplaybreaks[1]
\algrenewcommand\algorithmicrequire{\textbf{Input:}}
\algrenewcommand\algorithmicensure{\textbf{Output:}}
\newcommand{\R}{\mathbb R}
\newcommand{\eps}{\varepsilon}
\newcommand{\E}{\mathbb E}
\newcommand{\norm}[1]{\lVert #1\rVert}
\newcommand{\ip}[2]{\langle #1,#2\rangle}
\makeatletter
\renewcommand{\@fnsymbol}[1]{%
  \ifcase#1\or\textdagger\or\textasteriskcentered\else\@ctrerr\fi}
\makeatother

\title{Single-Loop Stochastic Projected Damped Extragradient Methods for Nonconvex--(Strongly) Concave Minimax Optimization}
\author{%
  Huiling Zhang\thanks{The first two authors contributed equally to this paper.}\\[-0.15em]
  {\small LSEC, ICMSEC, Academy of Mathematics and Systems Science,}\\[-0.15em]
  {\small Chinese Academy of Sciences, Beijing 100190, China, \texttt{zhanghl@amss.ac.cn}}\\[0.8em]
  Minhao Zhang\footnotemark[1]\\[-0.15em]
  {\small Department of Mathematics, College of Sciences, Shanghai University,}\\[-0.15em]
  {\small Shanghai 200444, China, \texttt{zhangminhao@shu.edu.cn}}\\[0.8em]
  Zi Xu\thanks{Corresponding author.}\\[-0.15em]
  {\small Department of Mathematics, College of Sciences, Shanghai University,}\\[-0.15em]
  {\small Shanghai 200444, China, \texttt{xuzi@shu.edu.cn}}\\[0.8em]
  Yu-Hong Dai\\[-0.15em]
  {\small LSEC, ICMSEC, Academy of Mathematics and Systems Science,}\\[-0.15em]
  {\small Chinese Academy of Sciences, Beijing 100190, China, \texttt{dyh@lsec.cc.ac.cn}}
}
\date{}

\begin{document}

\maketitle

\begin{abstract}
		We develop single-loop stochastic projected damped extragradient methods for nonconvex--(strongly) concave minimax optimization, with complexity guarantees for both game stationarity (GS) and optimization stationarity (OS). Our approach combines a stochastic projected damped extragradient (SPDE) method with a recursive variance-reduced variant, VR-SPDE, both of which retain a single-loop structure. Under an unbiased stochastic gradient oracle with uniformly bounded variance, SPDE finds an $\varepsilon$-game-stationary point with stochastic first-order oracle (SFO) complexities of $O(\kappa\varepsilon^{-4})$ and $O(\varepsilon^{-5})$ in the nonconvex--strongly concave and nonconvex--concave settings, respectively, where $\kappa=L/\mu$. Under an additional mean-square Lipschitz condition on the stochastic gradients, VR-SPDE improves these GS complexities to $O(\kappa^{3/2}\varepsilon^{-3})$ and $O(\varepsilon^{-9/2})$, respectively. For an $\varepsilon$-optimization-stationary point, SPDE achieves SFO complexities of $O(\kappa\varepsilon^{-4})$ and $O(\varepsilon^{-6})$, while VR-SPDE achieves $O(\kappa^{3/2}\varepsilon^{-3})$ and $O(\varepsilon^{-6})$, in the two settings, respectively. These OS guarantees match the best-known bounds achieved by multi-loop methods while preserving a single-loop implementation. To the best of our knowledge, our results provide the best-known SFO complexity guarantees among single-loop stochastic first-order methods for the respective stationarity criteria and problem classes.
	\end{abstract}

\medskip

\section{Introduction}\label{sec:introduction}
	
	We consider the stochastic minimax problem
	\begin{align}\label{eq:problem}
		\min_{x\in \mathcal X}\max_{y\in \mathcal Y}
		F(x,y):=\E_{\omega\sim\mathcal D}[f(x,y;\omega)],
	\end{align}
	where $\mathcal X\subseteq\R^n$ is closed and convex,
	$\mathcal Y\subseteq\R^p$ is compact and convex, and $F$ is smooth.
	The objective may be nonconvex in $x$ and is concave or strongly
	concave in $y$.
	The random variable $\omega$ follows an unknown distribution
	$\mathcal D$, and $\E$ denotes expectation.
	Problems of this form arise in distributed nonconvex
	optimization~\cite{Giannakis}, wireless
	systems~\cite{Chen20}, and statistical learning~\cite{Abadeh}.
	Their growing scope of applications motivates stochastic first-order
	methods that combine a simple iteration structure with strong
	oracle complexity guarantees.
	
	Two stationarity criteria are central to nonconvex minimax
	optimization: game stationarity (GS) and optimization stationarity
	(OS).
	Game stationarity concerns first-order conditions for the two
	players at a primal--dual pair, whereas optimization stationarity
	concerns the minimization of the primal value function $\phi(x):=\max_{y\in \mathcal Y}F(x,y).$
	These criteria capture different aspects of approximate solutions,
	and their complexity guarantees must be distinguished.
	This distinction is particularly relevant in the
	nonconvex--concave setting, where $\phi$ may be nonsmooth.
	We state the precise GS and OS criteria in
	Section~\ref{sec:problem} and analyze both criteria for the methods
	developed in this paper.
	
	For stochastic nonconvex--strongly concave problems, representative
	multi-loop methods obtain strong stochastic first-order oracle
	(SFO) complexity guarantees by approximately solving an inner
	maximization problem or a regularized saddle-point subproblem.
	Luo et al.~\cite{luo2020sreda} proposed SREDA, which combines nested
	ascent with SPIDER recursion and requires
	$O(\kappa^3\eps^{-3})$ stochastic gradient evaluations, where
	$\kappa=L/\mu$, $L$ is the gradient Lipschitz constant, and $\mu$
	is the strong concavity parameter.
	Zhang et al.~\cite{zhang2022sapdplus} developed SAPD+, which combines
	an inexact proximal-point outer scheme with an accelerated
	primal--dual inner solver.
	Their method requires $O(L\kappa\eps^{-4})$ oracle calls, while
	its variance-reduced variant requires $O(L\kappa^2\eps^{-3})$ calls.
	Among single-loop methods, Lin et al.~\cite{lin2020gda} analyzed
	two-timescale stochastic gradient descent--ascent (SGDA) and
	established an $O(\kappa^3\eps^{-4})$ stochastic gradient complexity
	for value-function stationarity.
	Huang et al.~\cite{huang2022accelerated} introduced Acc-MDA with
	a STORM-type estimator and obtained complexities of
	$\widetilde O(\kappa^{9/2}\eps^{-3})$ without a large batch and
	$\widetilde O(\kappa^{5/2}\eps^{-3})$ with a batch size of order
	$\kappa^4$.
	
	For stochastic nonconvex--concave problems, the possible
	nonsmoothness of the value function presents an additional
	challenge.
	Multi-loop methods often address this difficulty through proximal
	regularization or regularized saddle-point subproblems.
	Zhang et al.~\cite{zhang2022sapdplus} applied SAPD+ to this setting
	and established an $O(\eps^{-6})$ oracle complexity.
	Among single-loop methods, Lin et al.~\cite{lin2020gda} established
	an $O(\eps^{-8})$ stochastic gradient complexity for SGDA under
	value-function stationarity.
	For game stationarity, Zhang and
	Xu~\cite{zhangxu2025formda} proposed FORMDA, which combines
	vanishing dual regularization with recursive momentum and achieves
	an iteration complexity of $\widetilde O(\eps^{-13/2})$.
	
	These results highlight the challenge of obtaining strong
	complexity guarantees within a single-loop implementation.
	Multi-loop methods achieve sharper dependence on the target
	accuracy in several settings, but their outer iterations may
	require inner maximization, proximal, or regularized saddle-point
	subproblems to be solved to prescribed accuracies.
	The resulting inner-loop lengths and parameter schedules can
	depend on the target accuracy.
	This motivates the following question:
	\begin{quote}
		Can single-loop stochastic first-order methods achieve improved
		game-stationarity guarantees and match the best-known
		multi-loop optimization-stationarity guarantees for both
		nonconvex--strongly concave and nonconvex--concave problems?
	\end{quote}
	We address this question through a stochastic projected damped
	extragradient method and its recursive variance-reduced variant.
	Both methods retain a single-loop structure, and their analysis
	provides complexity guarantees for both GS and OS.
	
	\paragraph{Contributions.}
	We develop single-loop stochastic methods with complexity guarantees
	for both game and optimization stationarity.
	
	\begin{enumerate}[label=(\roman*),leftmargin=2.2em]
		\item \textbf{Single-loop algorithms.}
		We propose stochastic projected damped extragradient (SPDE)
		under an unbiased stochastic gradient oracle with uniformly
		bounded variance, and its recursive variance-reduced variant,
		VR-SPDE, under an additional mean-square Lipschitz condition.
		Both methods avoid nested solvers for maximization, proximal,
		or saddle-point subproblems.
		
		\item \textbf{Game-stationarity guarantees.}
		SPDE finds an $\eps$-game-stationary point with total SFO
		complexities of $O(\kappa\eps^{-4})$ and $O(\eps^{-5})$
		in the nonconvex--strongly concave and nonconvex--concave
		settings, respectively, with iteration complexities of
		$O(\sqrt{\kappa}\eps^{-2})$ and $O(\eps^{-5/2})$.
		Under the additional mean-square Lipschitz condition,
		VR-SPDE improves the GS oracle complexities to
		$O(\kappa^{3/2}\eps^{-3})$ and $O(\eps^{-9/2})$,
		respectively.
		
		\item \textbf{Optimization-stationarity guarantees.}
		SPDE finds an $\eps$-optimization-stationary point with
		$O(\kappa\eps^{-4})$ and $O(\eps^{-6})$ SFO calls in the
		nonconvex--strongly concave and nonconvex--concave settings,
		respectively.
		VR-SPDE achieves corresponding complexities of
		$O(\kappa^{3/2}\eps^{-3})$ and $O(\eps^{-6})$.
		These bounds match the best-known multi-loop dependence
		on target accuracy in both settings with a single-loop
		implementation.
	\end{enumerate}
	
	To the best of our knowledge, these are the best-known total SFO
	complexity guarantees among single-loop stochastic first-order
	methods for the respective criteria and problem classes.
	The same algorithms achieve both GS and OS guarantees:
	the stated SFO bounds coincide in the strongly concave setting,
	whereas in the merely concave setting, SPDE and VR-SPDE achieve
	GS bounds of $O(\eps^{-5})$ and $O(\eps^{-9/2})$, respectively,
	and a common OS bound of $O(\eps^{-6})$.
	Table~\ref{tab:results} compares representative results by
	stationarity criterion.
	\begin{table}[t]
		\centering
		\caption{Representative complexity guarantees for stochastic
			nonconvex minimax optimization.
			Bounds are total SFO complexities unless marked by
			$\dagger$, which denotes an iteration complexity.
			The dependence on $L$ is suppressed in this comparison.}
		\label{tab:results}
		\scriptsize
		\setlength{\tabcolsep}{10pt}
		\renewcommand{\arraystretch}{1.12}
		\begin{tabular}{@{}lccccc@{}}
			\toprule
			Method & Setting & Loop & VR & Criterion & Complexity \\
			\midrule
			SGDA~\cite{lin2020gda}
			& NC--SC & Single & No & OS
			& $O(\kappa^3\eps^{-4})$ \\
			
			Acc-MDA~\cite{huang2022accelerated}
			& NC--SC & Single & Yes & OS
			& $\widetilde O(\kappa^{9/2}\eps^{-3})$ \\
			
			SREDA~\cite{luo2020sreda}
			& NC--SC & Multi & Yes & OS
			& $O(\kappa^3\eps^{-3})$ \\
			
			SAPD+~\cite{zhang2022sapdplus}
			& NC--SC & Multi & No & OS
			& $O(\kappa\eps^{-4})$ \\
			
			VR-SAPD+~\cite{zhang2022sapdplus}
			& NC--SC & Multi & Yes & OS
			& $O(\kappa^2\eps^{-3})$ \\
			
			\textbf{SPDE (this paper)}
			& NC--SC & Single & No & GS/OS
			& $\boldsymbol{O(\kappa\eps^{-4})}$ \\
			
			\textbf{VR-SPDE (this paper)}
			& NC--SC & Single & Yes & GS/OS
			& $\boldsymbol{O(\kappa^{3/2}\eps^{-3})}$ \\
			
			\midrule
			SGDA~\cite{lin2020gda}
			& NC--C & Single & No & OS
			& $O(\eps^{-8})$ \\
			
			SAPD+~\cite{zhang2022sapdplus}
			& NC--C & Multi & No & OS
			& $O(\eps^{-6})$ \\
			
			\textbf{SPDE (this paper)}
			& NC--C & Single & No & OS
			& $\boldsymbol{O(\eps^{-6})}$ \\
			
			\textbf{VR-SPDE (this paper)}
			& NC--C & Single & Yes & OS
			& $\boldsymbol{O(\eps^{-6})}$ \\
			
			\midrule
			FORMDA~\cite{zhangxu2025formda}
			& NC--C & Single & Yes & GS
			& $\widetilde O(\eps^{-13/2})^{\dagger}$ \\
			
			\textbf{SPDE (this paper)}
			& NC--C & Single & No & GS
			& $\boldsymbol{O(\eps^{-5})}$ \\
			
			\textbf{VR-SPDE (this paper)}
			& NC--C & Single & Yes & GS
			& $\boldsymbol{O(\eps^{-9/2})}$ \\
			\bottomrule
		\end{tabular}
	\end{table}
	Here, NC--SC and NC--C denote the nonconvex--strongly concave
	and nonconvex--concave settings, respectively; ``VR'' indicates
	recursive variance reduction; and ``GS'' and ``OS'' denote game
	and optimization stationarity.
	The Acc-MDA bound is reported without a large batch.
	All bounds hold under the assumptions of the cited references,
	so comparisons should account for both the stationarity criterion
	and the stochastic oracle assumptions.

	\paragraph{Notation and Organization.}
	Throughout, $\norm{\cdot}$ and $\ip{\cdot}{\cdot}$ denote the Euclidean
	norm and inner product, and $\Pi_C$ and $N_C$ denote the projection
	onto and normal cone of a closed convex set $C$, respectively.
	The notation $O(\cdot)$ hides constants independent of the target
	accuracy $\eps$ and all explicitly displayed parameters;
	$\widetilde O(\cdot)$ also hides logarithmic factors.
	We write $a=\Theta(b)$ when $cb\le a\le Cb$ for some constants
	$c,C>0$ in the stated asymptotic regime, under the same convention.
	For a finite nonempty set $\mathcal I$,
	$J\sim\operatorname{Unif}\mathcal I$ denotes uniform sampling from $\mathcal I$.	
	
	The remainder of the paper is organized as follows.
	Section~\ref{sec:problem} covers the preliminaries and defines the GS and OS criteria.
	In Section~\ref{sec:plain}, we introduce SPDE and its stochastic oracle model,
	and establish GS complexity bounds for the nonconvex--strongly concave
	and nonconvex--concave settings through a unified convergence analysis.
	Section~\ref{sec:vr} is devoted to VR-SPDE, its paired oracle and recursive
	estimator, and the associated variance-reduced GS complexity bounds,
	with all oracle calls accounted for.
	OS guarantees for the original primal value function are established
	in Section~\ref{sec:os}.
	Finally, Section~\ref{sec:numerical} reports numerical experiments,
	and Section~\ref{sec:conclusion} concludes the paper.
	
	\section{Preliminaries and stationarity criteria}\label{sec:problem}
	We state the standing assumptions and define the two stationarity
	criteria. All product spaces carry the Euclidean norm.
	
	\begin{assumption}\label{ass:model}
		\begin{enumerate}[label=(\roman*),leftmargin=2.2em]
			\item
			$\mathcal X\subseteq\R^n$ is nonempty, closed, and convex;
			$\mathcal Y\subseteq\R^p$ is nonempty, compact, and convex, with
			$D_{\mathcal Y}:=\max_{y\in\mathcal Y}\norm{y}>0$.
			
			\item
			For almost every $\omega$, $f(\cdot,\cdot;\omega)$ is
			differentiable on a neighborhood of $\mathcal X\times\mathcal Y$.
			
			\item
			$F$ is continuously differentiable on a neighborhood of
			$\mathcal X\times\mathcal Y$, and for some $L>0$,
			\begin{align}\label{eq:smooth}
				\norm{\nabla F(x',y')-\nabla F(x,y)}
				\le L\norm{(x'-x,y'-y)}
			\end{align}
			for all $(x,y),(x',y')\in\mathcal X\times\mathcal Y$.
			
			\item
			$F(x,\cdot)$ is concave on $\mathcal Y$ for every
			$x\in\mathcal X$, and
			\begin{align}\label{eq:value}
				\phi(x):=\max_{y\in\mathcal Y}F(x,y),
				\qquad x\in\mathcal X,
				\qquad
				\phi_{\inf}:=\inf_{x\in\mathcal X}\phi(x)>-\infty.
			\end{align}
		\end{enumerate}
	\end{assumption}
	The maximum in~\eqref{eq:value} is attained by continuity and compactness.
	
	In the nonconvex--strongly concave setting, we additionally assume
	uniform $\mu$-strong concavity in $y$ for some $\mu>0$:
	\[
	F(x,y')
	\le F(x,y)+\ip{\nabla_y F(x,y)}{y'-y}
	-\frac{\mu}{2}\norm{y'-y}^2,
	\qquad x\in\mathcal X,\quad y,y'\in\mathcal Y.
	\]
	We write $\kappa:=L/\mu$. Stochastic oracle assumptions are specified
	with the corresponding algorithms.
	
	\paragraph{Projection and normal cone.}
	For a nonempty closed convex set $C\subseteq\R^d$, define
	\begin{align}\label{eq:normal}
		\Pi_C(w):=\argmin_{u\in C}\frac12\norm{u-w}^2,
		\quad
		N_C(u):=
		\begin{cases}
			\{\eta\in\R^d:\ip{\eta}{a-u}\le0\ \text{for all }a\in C\},
			& u\in C,\\
			\varnothing, & u\notin C.
		\end{cases}
	\end{align}
	The projection is unique and satisfies
	\begin{align}\label{eq:projection}
		u=\Pi_C(w)
		\quad\Longleftrightarrow\quad
		w-u\in N_C(u),
		\qquad
		\norm{\Pi_C(w)-\Pi_C(w')}\le\norm{w-w'}.
	\end{align}
	For nonempty $A$, let $\dist(v,A):=\inf_{a\in A}\norm{v-a}$.
	
	\paragraph{Game stationarity (GS).}
	For $(x,y)\in\mathcal X\times\mathcal Y$, define
	\begin{align}\label{eq:residual}
		\mathcal R(x,y)^2
		:=
		\dist^2\bigl(0,\nabla_x F(x,y)+N_{\mathcal X}(x)\bigr)
		+
		\dist^2\bigl(0,-\nabla_y F(x,y)+N_{\mathcal Y}(y)\bigr).
	\end{align}
	Thus $\mathcal R(x,y)=0$ exactly when
	\[
	0\in\nabla_x F(x,y)+N_{\mathcal X}(x),
	\qquad
	0\in-\nabla_y F(x,y)+N_{\mathcal Y}(y).
	\]
	These are the first-order conditions in $x$ and $y$; by concavity,
	the latter is equivalent to $y\in\arg\max_{v\in\mathcal Y}F(x,v)$.
	A random feasible pair $(x_{\rm out},y_{\rm out})$ is
	$\eps$-game-stationary in expectation if
	\begin{align}\label{eq:goal}
		\E\bigl[\mathcal R(x_{\rm out},y_{\rm out})^2\bigr]\le\eps^2,
	\end{align}
	which also implies
	$\E[\mathcal R(x_{\rm out},y_{\rm out})]\le\eps$.
	
	\paragraph{Optimization stationarity (OS).}
	To measure stationarity of the constrained primal problem, define
	\[
	\bar\phi(x):=
	\begin{cases}
		\phi(x), & x\in\mathcal X,\\
		+\infty, & x\notin\mathcal X.
	\end{cases}
	\]
	By Assumption~\ref{ass:model},
	$F(\cdot,y)+\frac{L}{2}\norm{\cdot}^2$ is convex on $\mathcal X$
	for each $y\in\mathcal Y$, so $\phi$ is $L$-weakly convex.
	Continuity of $F$ and compactness of $\mathcal Y$ also make $\phi$
	continuous relative to $\mathcal X$. Hence $\bar\phi$ is proper,
	lower semicontinuous, and $L$-weakly convex on $\R^n$.
	Define its Moreau envelope with parameter $1/(2L)$ and proximal point by
	\begin{align}\label{eq:os-original-envelope}
		p_0(z):=\min_{x\in\mathcal X}
		\left\{\phi(x)+L\norm{x-z}^2\right\},
		\qquad
		x_0^\star(z):=\argmin_{x\in\mathcal X}
		\left\{\phi(x)+L\norm{x-z}^2\right\}.
	\end{align}
	Since $1/(2L)<1/L$, standard Moreau-envelope theory ensures that
	$x_0^\star(z)$ exists uniquely for every $z\in\R^n$ and that $p_0$
	is continuously differentiable, with
	$\nabla p_0(z)=2L(z-x_0^\star(z))$.
	We define
	\begin{align}\label{eq:os-criterion}
		\mathcal S_{\rm OS}(z):=\norm{\nabla p_0(z)}.
	\end{align}
	A random output $z_{\rm out}\in\R^n$ is
	$\eps$-optimization-stationary in expectation if
	$\E[\mathcal S_{\rm OS}(z_{\rm out})^2]\le\eps^2$,
	which also implies $\E[\mathcal S_{\rm OS}(z_{\rm out})]\le\eps$.
	
	GS measures first-order residuals at a feasible primal--dual pair;
	OS measures the Moreau-envelope gradient at any $z\in\R^n$, with
	$x_0^\star(z)\in\mathcal X$.
	Their approximate guarantees are not interchangeable without further
	analysis, so we establish their complexity bounds separately.
	
	\section{A stochastic projected damped extragradient method}
	\label{sec:plain}
	We propose SPDE, a single-loop method combining projected
	predictor--corrector steps, damped dual momentum, and a moving
	primal proximal center.
	Unlike methods that solve regularized saddle-point subproblems
	in nested loops, SPDE updates all variables through a fixed
	sequence of projections and stochastic gradient queries.
	This section establishes GS guarantees; OS guarantees follow
	in Section~\ref{sec:os}.
	
	\paragraph{Regularization.}
	For $z\in\mathbb R^n$ and $0\le\tau\le L$, define
	\begin{align}
		g(x,y;z,\omega)
		&:=f(x,y;\omega)+L\norm{x-z}^2
		-\frac{\tau}{2}\norm{y}^2,\nonumber\\
		G(x,y;z)
		&:=\E_{\omega\sim\mathcal D}[g(x,y;z,\omega)]
		=F(x,y)+L\norm{x-z}^2
		-\frac{\tau}{2}\norm{y}^2.\label{eq:G}
	\end{align}
	By Assumption~\ref{ass:model}, $G$ is $L$-strongly convex in $x$
	and concave in $y$ for fixed $z$, with strong concavity modulus $\tau$ when $\tau>0$; the latter
	modulus becomes $\mu+\tau$ when $F(x,\cdot)$ is $\mu$-strongly
	concave. These properties hold for $G$, not necessarily for
	each sample loss $g$.
	Regularization guides the updates, while GS and OS are measured
	for the original problem by~\eqref{eq:residual}
	and~\eqref{eq:os-criterion}; the analysis accounts for the
	dual regularization bias.
	
	\paragraph{Stochastic gradients.}
	For $u=(x,y)\in\mathcal X\times\mathcal Y$ and an i.i.d.\ batch
	$\Omega=(\omega_1,\ldots,\omega_m)$ from $\mathcal D$, let
	\begin{align}\label{eq:batch-estimator-definition}
		\widehat{\nabla}F(u;\Omega)
		:=\frac1m\sum_{i=1}^m\nabla f(x,y;\omega_i).
	\end{align}
	Regularization gradients are computed exactly.
	Each iteration queries three fresh batches in sequence:
	\begin{align}\label{eq:three-estimators}
		\widehat g_t^{(0)}
		&:=\widehat{\nabla}F((x_t,y_t);\Omega_t^{(0)}),\quad
		\widehat g_t^{(1)}
		:=\widehat{\nabla}F(
		(\widetilde x_t,\widetilde y_t);\Omega_t^{(1)}),\quad
		\widehat g_t^{(2)}
		:=\widehat{\nabla}F(
		(x_{t+1},y_{t+1});\Omega_t^{(2)}).
	\end{align}
	We write
	$\widehat g_t^{(r)}
	=(\widehat g_{x,t}^{(r)},\widehat g_{y,t}^{(r)})$.
	Each batch is independent of all information available before
	its query.
	
	\paragraph{Single-loop structure.}
	SPDE maintains $(x_t,y_t)$, a proximal center $z_t$,
	normal-cone vectors $\xi_t,n_t$, and dual momentum $v_t$.
	Each iteration consists of four steps:
	\begin{enumerate}[label=(\roman*),leftmargin=2.2em]
		\item Use $\widehat g_t^{(0)}$ to form the projected predictor
		$(\widetilde x_t,\widetilde y_t)$.
		\item Use $\widehat g_t^{(1)}$ to form the intermediate momentum
		$\bar v_t$, the corrected iterate $(x_{t+1},y_{t+1})$,
		and its normal-cone vectors.
		\item Use $\widehat g_t^{(2)}$ at the corrected iterate to update
		the dual momentum.
		\item Move the proximal center via
		$z_{t+1}=(1-\beta)z_t+\beta x_{t+1}$.
	\end{enumerate}
	The distinctive coupling is between the normal-cone corrections
	and damped dual momentum: $n_{t+1}$ is scaled by $1+k$ and reused
	in the next momentum update.
	The fresh endpoint batch makes its gradient error conditionally
	mean zero given the corrected iterate and normal-cone vectors.
	Together with the moving center, these updates avoid separate
	inner and outer loops.
	
	SPDE adopts the fresh-sampling principle used in SGDA~\cite{lin2020gda}
	and SAPD+~\cite{zhang2022sapdplus}: each stochastic gradient estimate
	is computed from an independent batch. In each iteration, SPDE queries
	the current, predictor, and corrected points in sequence. The predictor
	anticipates the local descent--ascent direction, while the correction
	step reduces the error caused by evaluating the gradient at the current
	iterate. Damped dual momentum stabilizes the ascent dynamics, and the
	relaxed proximal-center update controls the primal regularization.
	By coupling these mechanisms in one loop, SPDE improves the stability
	and stationarity guarantees of existing fresh-sampling schemes without
	introducing an inner solver or an additional outer loop. Algorithm~\ref{alg:stochastic} gives the full procedure.
	The results below specify admissible parameters.
	Its GS output rule requires no residual evaluation;
	the OS output rule is given in Section~\ref{sec:os}.

	\begin{algorithm}[!ht]
		\caption{Stochastic projected damped extragradient (SPDE)}
		\label{alg:stochastic}
		\begin{algorithmic}[1]
			\Require
			Deterministic $x_0\in \mathcal X$ and $y_0\in \mathcal Y$;
			batch size $m\ge1$ and iteration budget $T\ge1$;
			$0\le\tau\le L$, $h>0$, $0<\alpha\le1$,
			$k\ge0$, and $0<\beta\le1$.
			\State
			Set $z_0=x_0$, $\xi_0=0\in\R^n$,
			and $n_0=v_0=0\in\R^p$.
			\For{$t=0,\ldots,T-1$}
			\State
			Draw a fresh batch $\Omega_t^{(0)}$ of $m$ samples
			and compute $\widehat g_t^{(0)}$
			by~\eqref{eq:three-estimators}.
			\State
			$\displaystyle
			\widetilde x_t
			=\Pi_{\mathcal X}\!\left(
			x_t-h\bigl(
			\widehat g_{x,t}^{(0)}
			+2L(x_t-z_t)+\xi_t
			\bigr)\right)$,\quad
			$\displaystyle
			\widetilde y_t
			=\Pi_{\mathcal Y}\!\left(
			y_t+h\bigl(
			\widehat g_{y,t}^{(0)}
			-\tau y_t-n_t+v_t
			\bigr)\right)$.
			
			\State
			Draw a fresh batch $\Omega_t^{(1)}$ of $m$ samples
			and compute $\widehat g_t^{(1)}$
			by~\eqref{eq:three-estimators}.
			\State
			$\displaystyle
			\bar v_t
			=\alpha v_t
			+k\bigl(
			\widehat g_{y,t}^{(1)}-\tau\widetilde y_t
			\bigr)$.
			\State
			$\displaystyle
			x_{t+1}
			=\Pi_{\mathcal X}\!\left(
			x_t-h\bigl(
			\widehat g_{x,t}^{(1)}
			+2L(\widetilde x_t-z_t)
			\bigr)\right)$,\quad
			$\displaystyle
			y_{t+1}
			=\Pi_{\mathcal Y}\!\left(
			y_t+h\bigl(
			\widehat g_{y,t}^{(1)}
			-\tau\widetilde y_t+\bar v_t
			\bigr)\right)$.
			\State
			$\displaystyle
			\xi_{t+1}
			=\frac{x_t-x_{t+1}}{h}
			-\widehat g_{x,t}^{(1)}
			-2L(\widetilde x_t-z_t)$,\quad
			$\displaystyle
			n_{t+1}
			=\frac{1}{1+k}\left(
			\frac{y_t-y_{t+1}}{h}
			+\widehat g_{y,t}^{(1)}
			-\tau\widetilde y_t+\bar v_t
			\right)$.
			
			\State
			Draw a fresh batch $\Omega_t^{(2)}$ of $m$ samples
			and compute $\widehat g_t^{(2)}$
			by~\eqref{eq:three-estimators}.
			\State
			$\displaystyle
			v_{t+1}
			=\alpha v_t
			+k\bigl(
			\widehat g_{y,t}^{(2)}
			-\tau y_{t+1}-n_{t+1}
			\bigr)$.
			\State
			$\displaystyle
			z_{t+1}
			=z_t+\beta(x_{t+1}-z_t)$.
			\EndFor
			\Ensure
			Draw $J\sim\operatorname{Unif}\{0,\ldots,T-1\}$
			and set $(x_{\rm out},y_{\rm out})
			=(x_{J+1},y_{J+1})$.
		\end{algorithmic}
	\end{algorithm}
	\subsection{Stochastic oracle and batch errors}
	
	\begin{assumption}\label{ass:oracle}
		Let $\mathcal A$ denote the pre-query sigma-algebra and
		$u=(x,y)\in\mathcal X\times\mathcal Y$ an
		$\mathcal A$-measurable query point.
		The oracle draws $\omega\sim\mathcal D$ independently of
		$\mathcal A$ and returns a measurable gradient satisfying
		\begin{align}\label{eq:oracle}
			\E[\nabla f(u;\omega)\mid\mathcal A]
			=\nabla F(u),
			\qquad
			\E\!\left[
			\norm{\nabla f(u;\omega)-\nabla F(u)}^2
			\,\middle|\,\mathcal A
			\right]\le\sigma^2,
		\end{align}
		where $\sigma\ge0$ is independent of the query point,
		iteration, and regularization parameter.
		Each batch consists of independent samples from $\mathcal D$,
		independent of the pre-query information.
		Each sample-gradient evaluation counts as one SFO call.
	\end{assumption}
	For a batch of size $m$, conditional independence and unbiasedness give
	\begin{align}\label{eq:batch}
		\E[\widehat{\nabla}F(u;\Omega)\mid\mathcal A]
		=\nabla F(u),
		\qquad
		\E\!\left[
		\norm{\widehat{\nabla}F(u;\Omega)-\nabla F(u)}^2
		\,\middle|\,\mathcal A
		\right]\le\frac{\sigma^2}{m}.
	\end{align}
	Smoothness and concavity are required only of $F$; individual
	sample losses need only satisfy the differentiability condition
	in Assumption~\ref{ass:model}.
	
	\paragraph{Oracle accounting.}
	SPDE reuses $\widehat g_t^{(1)}$ within iteration $t$ and draws
	a fresh batch for every subsequent query, including the first
	query of iteration $t+1$.
	With three batches of size $m$ per iteration, the algorithm uses
	$3mT$ SFO calls.
	Only the dual component of the third estimate is used;
	evaluating this component alone is still charged as one
	full-gradient call per sample, so $3mT$ remains a valid upper bound.
	
	\paragraph{Filtration and batch errors.}
	Let $\mathcal F_t$ contain the deterministic initialization and
	all samples drawn before iteration $t$, and define
	\[
	\mathcal F_t^{(1)}
	:=\sigma(\mathcal F_t,\Omega_t^{(0)}),
	\qquad
	\mathcal F_t^{(2)}
	:=\sigma(\mathcal F_t^{(1)},\Omega_t^{(1)}),
	\qquad
	\mathcal F_{t+1}
	:=\sigma(\mathcal F_t^{(2)},\Omega_t^{(2)}).
	\]
	The state $(x_t,y_t,z_t,\xi_t,n_t,v_t)$ is
	$\mathcal F_t$-measurable, the predictor
	$(\widetilde x_t,\widetilde y_t)$ is
	$\mathcal F_t^{(1)}$-measurable, and
	$(x_{t+1},y_{t+1},\xi_{t+1},n_{t+1},\bar v_t,z_{t+1})$
	is $\mathcal F_t^{(2)}$-measurable.
	
	Define
	\begin{align}\label{eq:noise}
		a_t
		&:=\widehat g_t^{(0)}-\nabla F(x_t,y_t),\quad
		b_t:=\widehat g_t^{(1)}
		-\nabla F(\widetilde x_t,\widetilde y_t),\quad
		c_t:=\widehat g_t^{(2)}
		-\nabla F(x_{t+1},y_{t+1}).
	\end{align}
	Thus
	\begin{align}\label{eq:sample-batch-identities}
		\widehat g_t^{(0)}
		&=\nabla F(x_t,y_t)+a_t,\quad
		\widehat g_t^{(1)}
		=\nabla F(\widetilde x_t,\widetilde y_t)+b_t,\quad
		\widehat g_t^{(2)}
		=\nabla F(x_{t+1},y_{t+1})+c_t.
	\end{align}
	By~\eqref{eq:batch},
	\begin{align}\label{eq:conditional-noise}
		\E[a_t\mid\mathcal F_t]=0,
		\qquad
		\E[b_t\mid\mathcal F_t^{(1)}]=0,
		\qquad
		\E[c_t\mid\mathcal F_t^{(2)}]=0,\\
		\E[\norm{a_t}^2\mid\mathcal F_t]\le\frac{\sigma^2}{m},
		\qquad
		\E[\norm{b_t}^2\mid\mathcal F_t^{(1)}]\le\frac{\sigma^2}{m},
		\qquad
		\E[\norm{c_t}^2\mid\mathcal F_t^{(2)}]\le\frac{\sigma^2}{m}.
	\end{align}
	Writing $c_t=(c_{x,t},c_{y,t})$, the endpoint momentum update becomes
	\[
	v_{t+1}
	=\alpha v_t
	+k\bigl(\nabla_y F(x_{t+1},y_{t+1})
	-\tau y_{t+1}-n_{t+1}\bigr)
	+k c_{y,t},
	\qquad
	\E[c_{y,t}\mid\mathcal F_t^{(2)}]=0.
	\]
	The errors need not be mutually independent because query points
	depend on earlier batches; only the conditional moments
	in~\eqref{eq:conditional-noise} are used.
	
	The deterministic arguments below concern the population objective.
	For stochastic updates, we substitute~\eqref{eq:sample-batch-identities},
	derive inequalities for the realized iterates, and then take
	conditional expectations.
	Mean-zero errors cancel only against factors measurable with
	respect to the corresponding pre-query sigma-algebra.

	\begin{lemma}\label{lem:feasibility}
		Suppose that Assumptions~\ref{ass:model}
		and~\ref{ass:oracle} hold, and let the iterates be generated
		by Algorithm~\ref{alg:stochastic}.
		Then, almost surely, for every $t=0,\ldots,T$,
		\begin{align}\label{eq:normal-membership}
			x_t,z_t\in \mathcal X,
			\qquad
			y_t\in \mathcal Y,
			\qquad
			\xi_t\in N_{\mathcal X}(x_t),
			\qquad
			n_t\in N_{\mathcal Y}(y_t).
		\end{align}
		The predictors satisfy
		$(\widetilde x_t,\widetilde y_t)\in \mathcal X\times \mathcal Y$
		for $t=0,\ldots,T-1$.
		Moreover,
		\begin{align}
			\E\!\left[
			\norm{x_t}^2+\norm{y_t}^2+\norm{z_t}^2
			+\norm{\xi_t}^2+\norm{n_t}^2+\norm{v_t}^2
			\right]&<\infty, \label{eq:state-second-moments} \\
			\E\!\left[
			\norm{\widetilde x_t}^2+\norm{\widetilde y_t}^2
			+\norm{\bar v_t}^2
			+\sum_{j=0}^{2}\norm{\widehat g_t^{(j)}}^2
			\right]&<\infty.\label{eq:auxiliary-second-moments}
		\end{align}
	\end{lemma}
	
	\begin{proof}
		The initialization is feasible, and the zero vector belongs
		to the normal cone at every point of a nonempty closed convex
		set.
		The projection steps ensure that the predictors and corrected
		iterates are feasible.
		Since $0<\beta\le1$,
		$
		z_{t+1}=(1-\beta)z_t+\beta x_{t+1}\in \mathcal X
		$
		whenever $z_t,x_{t+1}\in \mathcal X$.
		By the projection optimality condition
		in~\eqref{eq:projection}, the primal correction satisfies
		\begin{align*}
			x_t-h\bigl(
			\widehat g_{x,t}^{(1)}
			+2L(\widetilde x_t-z_t)
			\bigr)-x_{t+1}
			=h\xi_{t+1}
			\in N_{\mathcal X}(x_{t+1}).
		\end{align*}
		Similarly, the dual correction gives
		\begin{align*}
			y_t+h\bigl(
			\widehat g_{y,t}^{(1)}
			-\tau\widetilde y_t+\bar v_t
			\bigr)-y_{t+1}
			=h(1+k)n_{t+1}
			\in N_{\mathcal Y}(y_{t+1}).
		\end{align*}
		Normal cones are cones, and $h>0$ and $1+k>0$.
		Dividing by these positive scalars proves
		\eqref{eq:normal-membership} by induction.
		We next prove the second-moment statements. Fix
		$u_{\rm ref}\in \mathcal X\times \mathcal Y$. For every
		$u\in \mathcal X\times \mathcal Y$,~\eqref{eq:smooth},~\eqref{eq:batch}, and the
		tower property give
		\begin{align}
			\E\norm{\widehat{\nabla}F(u;\Omega)}^2
			&\le \E\norm{\nabla F(u)}^2+\frac{\sigma^2}{m}\le 2\norm{\nabla F(u_{\rm ref})}^2
			+2L^2\E\norm{u-u_{\rm ref}}^2
			+\frac{\sigma^2}{m}<\infty.
			\label{eq:batch-second-moment}
		\end{align}
		For any nonempty closed convex set $C$, fixed $u_C\in C$, and
		$w$,
		\begin{align}
			\E\norm{\Pi_C(w)}^2
			&\le 2\E\norm{\Pi_C(w)-u_C}^2+2\norm{u_C}^2\le 2\E\norm{w-u_C}^2+2\norm{u_C}^2<\infty.
			\label{eq:projection-second-moment}
		\end{align}
		Thus both a fresh batch evaluation and a projection preserve
		finite second moments.
		The deterministic initialization gives
		$
		\E\!\left[
		\norm{x_0}^2+\norm{y_0}^2+\norm{z_0}^2
		+\norm{\xi_0}^2+\norm{n_0}^2+\norm{v_0}^2
		\right]<\infty.
		$
		Suppose~\eqref{eq:state-second-moments} holds at iteration $t$.
		Applying~\eqref{eq:batch-second-moment} and
		\eqref{eq:projection-second-moment} in the order of
		Algorithm~\ref{alg:stochastic} yields
		\begin{align*}
			\E\norm{\widehat g_t^{(0)}}^2<\infty
			\ \Longrightarrow\ 
			\E\!\left[
			\norm{\widetilde x_t}^2+\norm{\widetilde y_t}^2
			\right]<\infty
			\ \Longrightarrow\ 
			\E\norm{\widehat g_t^{(1)}}^2<\infty.
		\end{align*}
		The affine updates and the two correction projections then give
		\begin{align*}
			\E\!\left[
			\norm{\bar v_t}^2
			+\norm{x_{t+1}}^2+\norm{y_{t+1}}^2
			+\norm{\xi_{t+1}}^2+\norm{n_{t+1}}^2
			+\norm{z_{t+1}}^2
			\right]<\infty.
		\end{align*}
		Finally,
		$
		\E\norm{\widehat g_t^{(2)}}^2<\infty
		\ \Longrightarrow\ 
		\E\norm{v_{t+1}}^2<\infty.
		$
		Hence~\eqref{eq:state-second-moments} holds at $t+1$, while
		all terms in~\eqref{eq:auxiliary-second-moments} are finite at
		$t$. Induction completes the proof.
	\end{proof}
	\subsection{Convergence analysis}\label{fc:section}
	We first define the regularized saddle point and value function,
	then construct a convergence potential.
	For $z\in\R^n$, let
	\begin{align}\label{eq:envelope}
		p(z)&:=\min_{x\in\mathcal X}\max_{y\in\mathcal Y}G(x,y;z),\quad
		(x^\star(z),y^\star(z))
		:=\text{the unique saddle point of }G(\cdot,\cdot;z).
	\end{align}
	Existence follows from coercivity in $x$, compactness of $\mathcal Y$,
	and the convex--concave minimax theorem; uniqueness follows from
	strong convexity in $x$ and strong concavity in $y$.
	
	\begin{lemma}\label{lem:sensitivity}
		Suppose that Assumption~\ref{ass:model} holds and $0<\tau\le L$. Then, for
		all $z,z'\in\R^n$,
		\begin{align}\label{eq:sensitivity}
			\norm{x^\star(z')-x^\star(z)}\le2\norm{z'-z},\qquad
			\norm{y^\star(z')-y^\star(z)}\le\sqrt{\frac{2L}{\tau}}\norm{z'-z},
		\end{align}
		and
		\begin{align}\label{eq:p-gradient}
			\nabla p(z)=2L(z-x^\star(z)),\qquad
			\norm{\nabla p(z')-\nabla p(z)}\le6L\norm{z'-z}.
		\end{align}
	\end{lemma}
	\begin{proof}
		Define
		$\mathcal T(x,y):=(\nabla_xF(x,y)+2Lx,
		-\nabla_yF(x,y)+\tau y)$. Strong convexity--strong concavity gives
		\begin{align*}
			\ip{\mathcal T(x',y')-\mathcal T(x,y)}{(x'-x,y'-y)}
			\ge L\norm{x'-x}^2+\tau\norm{y'-y}^2.
		\end{align*}
		Set $\delta_x:=x^\star(z')-x^\star(z)$ and
		$\delta_y:=y^\star(z')-y^\star(z)$, we have
		\begin{align*}
			\ip{\nabla_xG(x^\star(z),y^\star(z);z)}{\delta_x}
			+\ip{-\nabla_yG(x^\star(z),y^\star(z);z)}{\delta_y}&\ge0,\\
			\ip{\nabla_xG(x^\star(z'),y^\star(z');z')}{-\delta_x}
			+\ip{-\nabla_yG(x^\star(z'),y^\star(z');z')}{-\delta_y}&\ge0.
		\end{align*}
		Adding these inequalities and applying strong monotonicity give
		\begin{align*}
			L\norm{\delta_x}^2+\tau\norm{\delta_y}^2
			&\le
			\ip{\mathcal T(x^\star(z'),y^\star(z'))
				-\mathcal T(x^\star(z),y^\star(z))}{(\delta_x,\delta_y)}\\
			&\le2L\ip{z'-z}{\delta_x}
			\le\frac L2\norm{\delta_x}^2+2L\norm{z'-z}^2,
		\end{align*}
		which proves~\eqref{eq:sensitivity}.
		Let
		$\phi_\tau(x)=\max_{y\in \mathcal Y}\{F(x,y)-\tau\norm y^2/2\}$.
		Then $p(z)=\min_{x\in \mathcal X}\{\phi_\tau(x)+L\norm{x-z}^2\}$.
		Uniqueness of the minimizer and envelope differentiation give
		$\nabla p(z)=2L(z-x^\star(z))$.
		Together with~\eqref{eq:sensitivity}, this yields the Lipschitz bound $6L$.
	\end{proof}
	
	To combine the center estimates in Lemma~\ref{lem:sensitivity}
	with the primal--dual and momentum updates, we introduce the potential
	function $\mathcal V(z,w)$ below.
	For $w=(x,y,\xi,n,v)$ with $\xi\in N_{\mathcal X}(x),n\in N_{\mathcal Y}(y)$, define
	\begin{align}\label{eq:energy}
		\mathcal V(z,w):=p(z)-\inf_{z'\in\R^n}p(z')+H_z(w),
	\end{align}
	where $H_z(w):=E_z(w)+\frac{s}{256L^2}\norm v^2$, $s=\sqrt{2L\tau}$ and
	\begin{align*}
		E_z(w):={}&-\left(1-\frac{s}{16L}\right)[G(x,y;z)-p(z)]
		+\frac1L\norm{\nabla_xG(x,y;z)+\xi}^2\\
		&+\frac1L\norm{-\nabla_yG(x,y;z)+n-v}^2
		+\frac{s}{16L}\ip v{y-y^\star(z)}
		+\frac\tau{256}\norm{y-y^\star(z)}^2.
	\end{align*}
	We first establish the nonnegativity of $\mathcal V(z,w)$. 
	\begin{lemma}\label{lem:positive}
		Suppose that Assumption~\ref{ass:model} holds and $0<\tau\le L$. For any
		state as above, let $P=\nabla_xG(x,y;z)+\xi$ and
		$Q=-\nabla_yG(x,y;z)+n-v$. Then
		\begin{align}\label{eq:potential-lower}
			E_z(w)\ge{}&\frac1{2L}\norm P^2
			+\frac{s}{32}\norm{x-x^\star(z)}^2
			+\frac\tau2\norm{y-y^\star(z)}^2+\frac1L\left\|Q-\frac{s}{32}(y-y^\star(z))\right\|^2\ge0.
		\end{align}
		Consequently, $\mathcal V(z,w)\ge0$.
	\end{lemma}
	\begin{proof}
		Saddle-point optimality and strong convexity--strong  concavity give
		\begin{align*}
			G(x,y^\star(z);z)-p(z)&\ge\frac L2\norm{x-x^\star(z)}^2,\\
			G(x,y^\star(z);z)&\le G(x,y;z)
			-\ip{\nabla_yG(x,y;z)}{y-y^\star(z)}
			-\frac\tau2\norm{y-y^\star(z)}^2.
		\end{align*}		
		Adding these two inequalities and using
		$\ip{n}{y-y^\star(z)}\ge 0$, we obtain
		\begin{align}\label{eq:saddle-growth}
			\ip{-\nabla_yG(x,y;z)+n}{y-y^\star(z)}+G(x,y;z)-p(z)
			\ge\frac L2\norm{x-x^\star(z)}^2
			+\frac\tau2\norm{y-y^\star(z)}^2.
		\end{align}
		Define $Y=y-y^\star(z)$, combining the definition of $E_z(w)$ and~\eqref{eq:saddle-growth}, we have
		\begin{align*}
			E_z(w)={}&-\left(1-\frac{s}{16L}\right)(G(x,y;z)-p(z))+\frac1L\norm P^2
			+\frac{s}{16L}\ip{-\nabla_yG(x,y;z)+n}{Y}\\
			&+\frac1L\norm{Q-sY/32}^2+\frac\tau{512}\norm{Y}^2\\
			\ge{}&-(G(x,y;z)-p(z))+\frac1L\norm P^2
			+\frac{s}{32}\norm{x-x^\star(z)}^2
			+\frac1L\norm{Q-sY/32}^2.
		\end{align*}
		For
		$a\in \mathcal X$, strong convexity and the normal-cone condition give
		\begin{align*}
			G(a,y;z)&\ge G(x,y;z)+\ip P{a-x}+\frac L2\norm{a-x}^2
			\ge G(x,y;z)-\frac{\norm P^2}{2L},\\
			\min_{a\in \mathcal X}G(a,y;z)&\le G(x^\star(z),y;z)
			\le p(z)-\frac\tau2\norm{Y}^2,
		\end{align*}
		and hence
		$
		\frac{\norm P^2}{2L}-(G(x,y;z)-p(z))\ge\frac\tau2\norm{Y}^2.
		$
		Substituting this bound yields~\eqref{eq:potential-lower}. Finally,
		$\inf_zp(z)=\inf_x\phi_\tau(x)\ge\phi_{\inf}-\tau D_{\mathcal Y}^2/2>-\infty$.
		Thus $E_z$, $H_z$, and $\mathcal V$ are nonnegative.
	\end{proof}
	
	We next derive a one-step descent bound for $H_z$ with the center $z$ fixed. Consider feasible states
	$w=(x,y,\xi,n,v)$ and $w^+=(x^+,y^+,\xi^+,n^+,v^+)$ satisfying
	$\xi\in N_{\mathcal X}(x)$, $\xi^+\in N_{\mathcal X}(x^+)$,
	$n\in N_{\mathcal Y}(y)$, and $n^+\in N_{\mathcal Y}(y^+)$. For notational convenience, let $P$ and $Q$ be as defined in Lemma~\ref{lem:positive}, and define
	\begin{align}
		P^+&=\nabla_xG(x^+,y^+;z)+\xi^+,\quad
		Q^+=-\nabla_yG(x^+,y^+;z)+n^+-v^+,\nonumber\\
		X^+&=x^+-x^\star(z),\quad Y^+=y^+-y^\star(z).
		\label{fc:local-notation}
	\end{align}
	
	\begin{lemma}\label{fc:inexact-dissipation}
		Suppose that Assumption~\ref{ass:model} holds. Let $0<\tau\le L$, and set
		\begin{align}\label{eq:step-damping}
			L_G:=3L+\tau,\quad h:=\frac1{64L_G},\quad
			d:=hL_G=\frac1{64},\quad
			\alpha:=\left(1+\frac{hs}{16}\right)^{-1},\quad
			k:=\frac{h(8L-s)}{16+hs}.
		\end{align}
		Assume $v^+=\alpha v+k(\nabla_yG(x^+,y^+;z)-n^+)$.
		If for some $u\ge0$,
		\begin{align}\label{fc:inexact-error}
			\norm e\le2d\sqrt{\mathcal I}+u,
		\end{align}
		where $\mathcal I=\norm{P^+-P}^2+\norm{Q^+-Q}^2$, $e=(e_x,e_y)$ with $
		e_x=\frac{x-x^+}{h}-P^+$ and $
		e_y=\frac{y-y^+}{h}-Q^+$, then the following bound holds: 
		\begin{align}\label{fc:inexact-conclusion}
			H_z(w^+)-H_z(w)+\frac{hs}{32}H_z(w^+)
			\le-\mathcal D_z(w^+,w)+\frac{17}{L}u^2,
		\end{align}
		where 
		\begin{align}\label{fc:dissipation}
			\mathcal D_z(w^+,w)={}&\frac h4\norm{P^+}^2
			+\frac{3hs}{512L}\norm{Q^+}^2
			+\frac{7hs}{4096L}\norm{v^+}^2+\frac{h\tau s}{128}\norm{Y^+}^2\nonumber\\
			&+\frac{hs(16L-s)}{1024}\norm{X^+}^2
			+\frac1{4L}\mathcal I.
		\end{align}
	\end{lemma}
	\begin{proof}
		We first bound
		$E_z(w^+)-E_z(w)+(hs/32)E_z(w^+)$, then add the corresponding
		bound for the momentum term $s\norm v^2/(256L^2)$ to obtain the desired bound for $H_z$.
		For any quantity defined at both states, $\Delta$ denotes
		its value at $w^+$ minus its value at $w$. In particular,
		\begin{align}\label{fc:increments}
			\Delta x:=x^+-x,\quad \Delta y:=y^+-y,\quad \Delta v:=v^+-v,\quad
			\Delta P:=P^+-P,\quad \Delta Q:=Q^+-Q,
		\end{align}
		Let $\rho=(16L-s)/32$ and $Y=y-y^\star(z)$.
		To estimate the three parts of $E_z$ separately, write
		\begin{align}\label{fc:components}
			A=-\rho[G(x,y;z)-p(z)]+\frac12\norm P^2,\quad
			B=\frac12\norm Q^2,\quad
			C=\frac{s}{32}\ip{v}{Y}+\frac{L\tau}{512}\norm Y^2.
		\end{align}
		Then $(L/2)E_z(w)=A+B+C$. We first analyze the changes in $A$ and $B$. Define $R:=Q+v$ and $R^+:=Q^++v^+$.
		
		Since $\xi\in N_{\mathcal X}(x)$ and $n^+\in N_{\mathcal Y}(y^+)$,
		$\ip\xi{\Delta x}\le0$ and $\ip{n^+}{\Delta y}\ge0$.
		At the mixed point $(x^+,y)$, applying $L$-strong convexity and $\tau$-strong concavity gives
		\begin{align*}
			G(x^+,y;z)&\ge G(x,y;z)+\ip P{\Delta x}
			+\frac L2\norm{\Delta x}^2,\\
			G(x^+,y;z)&\le G(x^+,y^+;z)+\ip{R^+}{\Delta y}
			-\frac\tau2\norm{\Delta y}^2.
		\end{align*}
		Consequently,
		\begin{align}\label{fc:value-cross}
			G(x^+,y^+;z)-G(x,y;z)\ge\ip P{\Delta x}-\ip{R^+}{\Delta y}
			+\frac L2\norm{\Delta x}^2
			+\frac\tau2\norm{\Delta y}^2.
		\end{align}
		By the strong convexity--strong concavity of $G$, we have
		\begin{align*}
			&\big\langle\nabla_xG(x^+,y^+;z)-\nabla_xG(x,y;z),\Delta x\big\rangle+\big\langle-\nabla_yG(x^+,y^+;z)+\nabla_yG(x,y;z),\Delta y\big\rangle\\
			&\ge L\|\Delta x\|^2+\tau\|\Delta y\|^2,
		\end{align*}
		By the monotonicity of \(N_{\mathcal X}\) and \(N_{\mathcal Y}\), we obtain
		$\ip{\Delta\xi}{\Delta x}\ge0$ and $\ip{\Delta n}{\Delta y}\ge0$.
		Hence
		\begin{align}\label{fc:monotonicity}
			\ip{\Delta P}{\Delta x}+\ip{\Delta R}{\Delta y}
			\ge L\norm{\Delta x}^2+\tau\norm{\Delta y}^2.
		\end{align}
		By the definition of $A$,~\eqref{fc:value-cross}, and $\rho>0$, we get
		\begin{align}
			\Delta A
			\le{}&\frac12(\norm{P^+}^2-\norm P^2)
			-\rho\ip P{\Delta x}+\rho\ip{R^+}{\Delta y}
			-\frac{L\rho}{2}\norm{\Delta x}^2
			-\frac{\rho\tau}{2}\norm{\Delta y}^2\nonumber\\
			\le&-\rho\ip{P^+}{\Delta x}
			+\rho\ip{R^+}{\Delta y}+\ip{P^+}{\Delta P}
			-\frac12\norm{\Delta P-\rho\Delta x}^2
			-\frac{\rho\tau}{2}\norm{\Delta y}^2.
			\label{fc:A-increment}
		\end{align}
		Moreover, 
		$\Delta B=\ip{Q^+}{\Delta Q}-\norm{\Delta Q}^2/2$.
		We next compute $\Delta C/h$. By the definition of $C$, we have
		\begin{align}
			\frac{\Delta C}{h}
			={}&\frac{s}{32}\ip{\Delta v/h}{Y^+}
			+\frac{s}{32}\ip{v^+}{\Delta y/h}
			-\frac{s}{32h}\ip{\Delta v}{\Delta y}
			+\frac{L\tau}{256}\ip{Y^+}{\Delta y/h}
			-\frac{L\tau}{512h}\norm{\Delta y}^2.\label{lem4:deltac}
		\end{align}
		The identities $\alpha^{-1}=1+hs/16$ and
		$k/\alpha=h(8L-s)/16$ from~\eqref{eq:step-damping} give
		\begin{align}
			\frac{\Delta v}{h}
			&=\frac{1-\alpha^{-1}}h v^+-\frac{k}{\alpha h}R^+
			=-\frac{s}{16}v^+-\frac{8L-s}{16}R^+
			=-\frac L2R^++\frac{s}{16}Q^+.
			\label{fc:momentum-dynamics}
		\end{align}
		Substituting this identity, $\Delta y=-h(Q^++e_y)$,
		and $v^+=R^+-Q^+$ into \eqref{lem4:deltac}, we have
		\begin{align}
			\frac{\Delta C}{h}={}&
			-\frac{s}{32}\ip{R^+}{Q^++e_y}
			+\frac{s}{32}\norm{Q^+}^2+\frac{s}{32}\ip{Q^+}{e_y}-\frac{Ls}{64}\ip{R^+}{Y^+}\nonumber\\
			&
			-\frac{L\tau}{256}\ip{Y^+}{e_y}
			-\frac{s}{32h}\ip{\Delta v}{\Delta y}
			-\frac{L\tau}{512h}\norm{\Delta y}^2.
			\label{fc:C-increment}
		\end{align}
		We now combine $\Delta A$ and $\Delta B$.
		Adding the formula for $\Delta B$ to~\eqref{fc:A-increment} and using
		$\Delta x=-h(P^++e_x)$ and $\Delta y=-h(Q^++e_y)$ yields
		\begin{align}
			\frac{\Delta A+\Delta B}{h}\le{}&
			\rho\ip{P^+}{P^++e_x}-\rho\ip{R^+}{Q^++e_y}
			+\frac1h\bigl(\ip{P^+}{\Delta P}+\ip{Q^+}{\Delta Q}\bigr)\nonumber\\
			&-\frac1{2h}\norm{\Delta P-\rho\Delta x}^2
			-\frac1{2h}\norm{\Delta Q}^2-\frac{\rho\tau}{2h}\norm{\Delta y}^2.
			\label{fc:AB-start}
		\end{align}
		Using
		$\Delta R=\Delta Q+\Delta v$ and~\eqref{fc:monotonicity}, we have
		\begin{align}
			\frac1h\bigl(\ip{P^+}{\Delta P}+\ip{Q^+}{\Delta Q}\bigr)
			\le{}&-L\norm{P^++e_x}^2-\tau\norm{Q^++e_y}^2\nonumber\\
			&-\frac1h\bigl(\ip{e_x}{\Delta P}+\ip{e_y}{\Delta Q}\bigr)
			-\frac1h\ip{Q^++e_y}{\Delta v}.
			\label{fc:monotonicity-use}
		\end{align}
		Substituting~\eqref{fc:monotonicity-use} and~\eqref{fc:momentum-dynamics}
		into~\eqref{fc:AB-start}, and using $L/2-\rho=s/32$, we get
		\begin{align}
			\frac{\Delta A+\Delta B}{h}\le{}&
			\rho\norm{P^+}^2-L\norm{P^++e_x}^2
			-\tau\norm{Q^++e_y}^2+\rho\ip{P^+}{e_x}+\frac{s}{32}\ip{R^+}{Q^++e_y}\nonumber\\
			&
			-\frac{s}{16}\ip{Q^+}{Q^++e_y}
			-\frac1h\bigl(\ip{e_x}{\Delta P}+\ip{e_y}{\Delta Q}\bigr)-\frac1{2h}\norm{\Delta P-\rho\Delta x}^2-\frac1{2h}\norm{\Delta Q}^2
			\nonumber\\
			&
			-\frac{\rho\tau}{2h}\norm{\Delta y}^2.
			\label{fc:AB-increment}
		\end{align}
		We can now estimate the weighted change of $E_z$.
		Since $(L/2)E_z=A+B+C$, we have
		\begin{align*}
			\frac{L}{2h}\left[E_z(w^+)-E_z(w)+\frac{hs}{32}E_z(w^+)\right]
			=\frac{\Delta A+\Delta B+\Delta C}{h}+\frac{Ls}{64}E_z(w^+),
		\end{align*}
		and
		\begin{align}
			\frac{Ls}{64}E_z(w^+)={}&-\frac{s(16L-s)}{1024}[G(x^+,y^+;z)-p(z)]+\frac{s}{64}(\norm{P^+}^2+\norm{Q^+}^2)\nonumber\\
			&+\frac{L\tau}{512}\ip{v^+}{Y^+}
			+\frac{L\tau s}{16384}\norm{Y^+}^2.
			\label{fc:endpoint-energy}
		\end{align}
		Combining $v^+=R^+-Q^+$, ~\eqref{eq:saddle-growth}, \eqref{fc:C-increment}, \eqref{fc:AB-increment} and \eqref{fc:endpoint-energy}, we obtain
		\begin{align}\label{fc:decomposition}
			\frac{L}{2h}\left[E_z(w^+)-E_z(w)
			+\frac{hs}{32}E_z(w^+)\right]
			\le\mathcal C+\mathcal J,
		\end{align}
		where
		\begin{align}
			\mathcal C={}&\frac{32L-s}{64}\norm{P^+}^2
			-L\norm{P^++e_x}^2-\tau\norm{Q^++e_y}^2
			-\frac{s}{64}\norm{Q^+}^2+\rho\ip{P^+}{e_x}-\frac{s}{32}\ip{Q^+}{e_y}\nonumber\\
			&
			-\frac{L\tau}{256}\ip{Y^+}{e_y}
			-\frac{L\tau}{512}\ip{Q^+}{Y^+}+\left(\frac{L\tau s}{16384}
			-\frac{\tau s(16L-s)}{2048}\right)\norm{Y^+}^2
			-\frac{Ls(16L-s)}{2048}\norm{X^+}^2,
			\label{fc:current-block}\\
			\mathcal J={}&-\frac1{2h}\norm{\Delta P-\rho\Delta x}^2
			-\frac1{2h}\norm{\Delta Q}^2
			-\left(\frac{\rho\tau}{2h}+\frac{L\tau}{512h}\right)
			\norm{\Delta y}^2-\frac{s}{32h}\ip{\Delta v}{\Delta y}\nonumber\\
			&
			-\frac1h\bigl(\ip{e_x}{\Delta P}+\ip{e_y}{\Delta Q}\bigr).
			\label{fc:increment-block}
		\end{align}
		We first bound $\mathcal C$. By $s\le2L$ and Young's inequality, we have
		\begin{align}
			&\frac{32L-s}{64}\norm{P^+}^2
			-L\norm{P^++e_x}^2+\rho\ip{P^+}{e_x}\nonumber\\
			&=-\frac{32L+s}{64}\norm{P^+}^2
			-\frac{48L+s}{32}\ip{P^+}{e_x}-L\norm{e_x}^2\le-\frac L4\norm{P^+}^2+4L\norm{e_x}^2.
			\label{fc:primal-current-bound}
		\end{align}
		The three remaining cross terms in~\eqref{fc:current-block} satisfy
		\begin{align}
			\frac{s}{32}|\ip{Q^+}{e_y}|
			&\le\frac{s}{256}\norm{Q^+}^2+\frac{s}{16}\norm{e_y}^2,
			\quad
			\frac{L\tau}{512}|\ip{Q^+}{Y^+}|
			\le\frac{s}{256}\norm{Q^+}^2
			+\frac{L\tau s}{8192}\norm{Y^+}^2,
			\label{fc:young-qy}\\
			\frac{L\tau}{256}|\ip{Y^+}{e_y}|
			&\le\frac{L\tau s}{16384}\norm{Y^+}^2
			+\frac{s}{32}\norm{e_y}^2.
			\label{fc:young-ye}
		\end{align}
		Substituting~\eqref{fc:primal-current-bound}--\eqref{fc:young-ye}
		into~\eqref{fc:current-block}, we get
		\begin{align}
			\mathcal C\le{}&-\frac L4\norm{P^+}^2
			-\frac{s}{128}\norm{Q^+}^2
			-\frac{L\tau s}{256}\norm{Y^+}^2-\frac{Ls(16L-s)}{2048}\norm{X^+}^2+4L\norm e^2.
			\label{fc:current-bound}
		\end{align}
		We next bound $\mathcal J$, 
		$\norm{a-b}^2\ge\norm a^2/2-\norm b^2$ and $\Delta x=-h(P^++e_x)$ give
		\begin{align}
			&-\frac1{2h}\norm{\Delta P-\rho\Delta x}^2
			-\frac1{2h}\norm{\Delta Q}^2\le-\frac{\mathcal I}{4h}
			+h\rho^2(\norm{P^+}^2+\norm{e_x}^2).
			\label{fc:negative-squares}
		\end{align}
		Using $\Delta v=-\Delta(\nabla_yG)+\Delta n-\Delta Q$,
		$\ip{\Delta n}{\Delta y}\ge0$, and $L_G$-smoothness, we obtain
		\begin{align}
			-\frac{s}{32h}\ip{\Delta v}{\Delta y}
			&\le\frac{L_Gs}{32h}
			\sqrt{\norm{\Delta x}^2+\norm{\Delta y}^2}\norm{\Delta y}
			+\frac{s}{32h}\norm{\Delta Q}\norm{\Delta y}\nonumber\\
			&\le\frac{hL_Gs}{32}(\norm{P^+}^2+\norm{e_x}^2)
			+\left(\frac{3hL_Gs}{32}+\frac{hL\tau}{32}\right)
			(\norm{Q^+}^2+\norm{e_y}^2)
			+\frac{\mathcal I}{32h}.
			\label{fc:feedback-increment}
		\end{align}
		For the remaining error inner products,~\eqref{fc:inexact-error} gives
		\begin{align}
			-\frac1h\bigl(\ip{e_x}{\Delta P}+\ip{e_y}{\Delta Q}\bigr)
			&\le\frac{\norm e\sqrt{\mathcal I}}h
			\le\frac{2d+1/32}{h}\mathcal I+\frac8h u^2,
			\label{fc:error-inner-product}
		\end{align}
		Substituting~\eqref{fc:negative-squares},~\eqref{fc:feedback-increment},
		and~\eqref{fc:error-inner-product} into~\eqref{fc:increment-block}, and discarding the nonpositive $\norm{\Delta y}^2$ term, yields
		\begin{align}
			\mathcal J\le{}&\frac{5Ld}{16}\norm{P^+}^2
			+\frac{ds}{8}\norm{Q^+}^2+\frac L2\norm e^2-\frac{7/32-2d-1/32}{h}\mathcal I+\frac8h u^2.
			\label{fc:increment-bound}
		\end{align}
		Adding~\eqref{fc:current-bound} and~\eqref{fc:increment-bound}
		yields the following bound on the right-hand side of~\eqref{fc:decomposition}:
		\begin{align}
			\mathcal C+\mathcal J\le{}&-\frac{(4-5d)L}{16}\norm{P^+}^2
			-\frac{(1-16d)s}{128}\norm{Q^+}^2
			-\frac{L\tau s}{256}\norm{Y^+}^2\nonumber\\
			&-\frac{Ls(16L-s)}{2048}\norm{X^+}^2
			+\frac{9L}{2}\norm e^2
			-\frac{7/32-2d-1/32}{h}\mathcal I+\frac8h u^2.
			\label{fc:before-absorption}
		\end{align}
		By~\eqref{fc:inexact-error}, we have 
		$
		\norm e^2\le(2d\sqrt{\mathcal I}+u)^2\le8d^2\mathcal I+2u^2.
		$
		Substituting this bound into~\eqref{fc:before-absorption}
		and using $hL\le d=1/64$, then 
		multiplying~\eqref{fc:decomposition} by $2h/L$, we get
		\begin{align}
			&E_z(w^+)-E_z(w)+\frac{hs}{32}E_z(w^+)\nonumber\\
			&\le-\frac h4\norm{P^+}^2
			-\frac{hs}{128L}\norm{Q^+}^2
			-\frac{h\tau s}{128}\norm{Y^+}^2
			-\frac{hs(16L-s)}{1024}\norm{X^+}^2-\frac1{4L}\mathcal I
			+\left(18h+\frac{16}{L}\right)u^2.
			\label{fc:base-dissipation}
		\end{align}
		Since $hL\le 1/64$, the coefficient $18h+16/L$ in~\eqref{fc:base-dissipation} is bounded above by $17/L$. To pass from $E_z$ to $H_z$, we write
		\begin{align}
			&H_z(w^+)-H_z(w)+\frac{hs}{32}H_z(w^+)\nonumber\\
			&=E_z(w^+)-E_z(w)+\frac{hs}{32}E_z(w^+)+\frac{s}{256L^2}\left(\norm{v^+}^2-\norm v^2
			+\frac{hs}{32}\norm{v^+}^2\right).\label{lem4:hz}
		\end{align}
		By~\eqref{fc:momentum-dynamics}, $R^+=Q^++v^+$,
		and Young's inequality, we have
		\begin{align}
			\norm{v^+}^2-\norm v^2
			\le-\frac{hL}{2}\norm{v^+}^2+\frac{hL}{2}\norm{Q^+}^2.
			\label{fc:velocity-energy}
		\end{align}
		Using~\eqref{fc:velocity-energy} and $s\le2L$, we obtain
		\begin{align}
			&\frac{s}{256L^2}(\norm{v^+}^2-\norm v^2)
			+\frac{hs}{32}\frac{s}{256L^2}\norm{v^+}^2
			\le\frac{hs}{512L}\norm{Q^+}^2-\frac{7hs}{4096L}\norm{v^+}^2.
			\label{fc:velocity-dissipation}
		\end{align}
		Substituting~\eqref{fc:base-dissipation} and~\eqref{fc:velocity-dissipation}
		into \eqref{lem4:hz}, by the definition~\eqref{fc:dissipation}, we complete the proof.
	\end{proof}
	
	We next bound the change in $H_z(w)$ induced by moving the center while keeping the state $w$ fixed.
	\begin{lemma}\label{lem:move}
		Suppose that Assumption~\ref{ass:model} holds and $0<\tau\le L$.
		For $w=(x,y,\xi,n,v)$ satisfying~\eqref{eq:normal-membership}
		and any $z'=z+\delta$,
		\begin{align}\label{eq:moving}
			H_{z'}(w)-H_z(w)\le{}&
			\left[2L\norm{x-x^\star(z)}+4\norm{\nabla_xG(x,y;z)+\xi}
			+\frac18\norm v+\frac{s}{128}\norm{y-y^\star(z)}\right]\norm\delta\nonumber\\
			&+\frac{1025L}{128}\norm\delta^2.
		\end{align}
	\end{lemma}
	\begin{proof}
		By \eqref{eq:G} and~\eqref{eq:p-gradient}, we have
		\begin{align*}
			G(x,y;z')-G(x,y;z)&=-2L\ip{x-z}\delta+L\norm\delta^2,\quad
			|p(z')-p(z)-\ip{\nabla p(z)}\delta|\le3L\norm\delta^2.
		\end{align*}
		Using $2L(x-z)=2L(x-x^\star(z))-\nabla p(z)$ and
		$\nabla_xG(x,y;z')=\nabla_xG(x,y;z)-2L\delta$, we obtain
		\begin{align*}
			-\left(1-\frac{s}{16L}\right)
			\bigl[G(x,y;z')-p(z')-G(x,y;z)+p(z)\bigr]
			&\le2L\norm{x-x^\star(z)}\norm\delta+4L\norm\delta^2,\\
			\frac1L\left(\norm{\nabla_xG(x,y;z')+\xi}^2-\norm{\nabla_xG(x,y;z)+\xi}^2\right)
			&\le4\norm{\nabla_xG(x,y;z)+\xi}\norm\delta+4L\norm\delta^2.
		\end{align*}
		For the remaining center-dependent terms,~\eqref{eq:sensitivity} gives
		\begin{align*}
			\frac{s}{16L}|\ip v{y^\star(z)-y^\star(z')}|&\le\frac18\norm v\norm\delta,\\
			\frac\tau{256}\left(\norm{y-y^\star(z')}^2-\norm{y-y^\star(z)}^2\right)
			&\le\frac{s}{128}\norm{y-y^\star(z)}\norm\delta+\frac L{128}\norm\delta^2.
		\end{align*}
		Summing the four bounds above therefore
		gives~\eqref{eq:moving}.
	\end{proof}
	
	Define $w_t=(x_t,y_t,\xi_t,n_t,v_t)$ and the virtual state using $\nabla_yF$ instead of $\widehat{\nabla}_yF$,
	\begin{align}\label{eq:virtual}
		v_{t+1}^{0}:=\alpha v_t+k\bigl(\nabla_y F(x_{t+1},y_{t+1})
		-\tau y_{t+1}-n_{t+1}\bigr),\quad
		w_{t+1}^{0}:=(x_{t+1},y_{t+1},\xi_{t+1},n_{t+1},v_{t+1}^{0}).
	\end{align}
	Apply Lemma~\ref{fc:inexact-dissipation} with
	$(z,w,w^+)=(z_t,w_t,w_{t+1}^{0})$, so that $v^+=v_{t+1}^{0}$.
	For the batch $\Omega_t^{(2)}=(\omega_{t,i}^{(2)})_{i=1}^m$, the actual update satisfies
	\begin{align}
		v_{t+1}&=v_{t+1}^{0}+k(\widehat{\nabla}_yF((x_{t+1},y_{t+1});\Omega_t^{(2)})
		-\nabla_yF(x_{t+1},y_{t+1})).\label{eq:virtual-noise}
	\end{align}
	
	\begin{lemma}\label{lem:virtual-error}
		Suppose that Assumption~\ref{ass:model} holds and
		Algorithm~\ref{alg:stochastic} uses~\eqref{eq:step-damping}.
		At center $z_t$, apply~\eqref{fc:local-notation}--\eqref{fc:dissipation}
		to $(w_t,w_{t+1}^{0})$, and denote the corresponding quantities by $\mathcal I_t^{0}$ and $\mathcal D_t^{0}$. For each realization of $\Omega_t^{(0)}$ and $\Omega_t^{(1)}$, define
		\begin{align}\label{eq:sample-U}
			U_t(\Omega_t^{(0)},\Omega_t^{(1)}):=2d\norm{a_t}+2\norm{b_t},\quad U_t:=U_t(\Omega_t^{(0)},\Omega_t^{(1)}),
		\end{align}
		where $a_t$ and $b_t$ are the sample errors in~\eqref{eq:noise}.
		Then
		\begin{align}\label{eq:virtual-fixed}
			H_{z_t}(w_{t+1}^{0})-H_{z_t}(w_t)
			\le-\mathcal D_t^{0}+\frac{17}{L}U_t^2.
		\end{align}
	\end{lemma}
	\begin{proof}
		Fix $\Omega_t^{(0)},\Omega_t^{(1)}$.
		We verify~\eqref{fc:inexact-error} for the pair $(w_t,w_{t+1}^{0})$ with $u=U_t$. The correction step and~\eqref{eq:virtual} give
		\begin{align}
			x_{t+1}-x_t&=-h(\nabla_xG(x_{t+1},y_{t+1};z_t)+\xi_{t+1}+e_{x,t}),\\
			y_{t+1}-y_t&=-h(-\nabla_yG(x_{t+1},y_{t+1};z_t)+n_{t+1}-v_{t+1}^{0}+e_{y,t}),\label{eq:virtual-dynamics}\\
			e_{x,t}&=\nabla_xG(\widetilde x_t,\widetilde y_t;z_t)
			-\nabla_xG(x_{t+1},y_{t+1};z_t)+b_{x,t},\\
			e_{y,t}&=(1+k)\bigl[\nabla_yG(x_{t+1},y_{t+1};z_t)
			-\nabla_yG(\widetilde x_t,\widetilde y_t;z_t)-b_{y,t}\bigr].\label{eq:virtual-errors}
		\end{align}
		Set $e_t=(e_{x,t},e_{y,t})$ and $r=d(1+k)$.
		Projection nonexpansiveness and~\eqref{eq:virtual-errors} yield
		\begin{align}
			\norm{e_t}
			&\le (1+k)L_G\norm{(\widetilde x_t-x_{t+1},\widetilde y_t-y_{t+1})}
			+(1+k)\norm{b_t}\nonumber\\
			&\le r\left(\sqrt{\mathcal I_t^{0}}+\norm{e_t}+\norm{a_t}\right)
			+(1+k)\norm{b_t}\le2d\sqrt{\mathcal I_t^{0}}+U_t,\label{eq:error-U}
		\end{align}
		where the last inequality holds by $r/(1-r)\le2d$ and $(1+k)/(1-r)\le2$.
		Applying Lemma~\ref{fc:inexact-dissipation} gives~\eqref{eq:virtual-fixed} for the virtual state.
	\end{proof}
	
	Set $\mathcal V_t=\mathcal V(z_t,w_t)$.
	Combining Lemmas~\ref{lem:move} and~\ref{lem:virtual-error}
	yields the following joint descent bound.
	
	\begin{lemma}\label{lem:virtual-joint}
		Suppose that Assumption~\ref{ass:model} holds and
		Algorithm~\ref{alg:stochastic} uses~\eqref{eq:step-damping} with
		\begin{align}\label{eq:center-relaxation}
			\beta:=\frac{hs}{4096}.
		\end{align}
		Then, for every realization of $\Omega_t^{(0)},\Omega_t^{(1)}$,
		\begin{align}\label{eq:virtual-joint}
			\mathcal V(z_{t+1},w_{t+1}^{0})-\mathcal V_t
			\le-\frac12\mathcal D_t^{0}
			-\frac{\beta}{8L}\norm{\nabla p(z_t)}^2+\frac{17}{L}U_t^2.
		\end{align}
	\end{lemma}
	\begin{proof}
		By~\eqref{eq:energy}, we decompose the change in the potential as
		\begin{align}
			\mathcal V(z_{t+1},w_{t+1}^{0})-\mathcal V_t
			={}&H_{z_t}(w_{t+1}^{0})-H_{z_t}(w_t)+H_{z_{t+1}}(w_{t+1}^{0})-H_{z_t}(w_{t+1}^{0})
			+p(z_{t+1})-p(z_t).\label{lem7:vt}
		\end{align}
		The first difference is bounded by~\eqref{eq:virtual-fixed}.
		To bound the remaining terms, let $X_t^+=x_{t+1}-x^\star(z_t)$,
		$Y_t^+=y_{t+1}-y^\star(z_t)$,
		$P_t^+=\nabla_xG(x_{t+1},y_{t+1};z_t)+\xi_{t+1}$, and
		$\delta_t:=z_{t+1}-z_t$.
		By~\eqref{eq:moving}, the center change satisfies
		\begin{align*}
			&H_{z_{t+1}}(w_{t+1}^{0})-H_{z_t}(w_{t+1}^{0})
			\le\left[2L\norm{X_t^+}+4\norm{P_t^+}
			+\frac18\norm{v_{t+1}^{0}}+\frac{s}{128}\norm{Y_t^+}\right]
			\norm{\delta_t}+\frac{1025L}{128}\norm{\delta_t}^2.
		\end{align*}
		The center update and~\eqref{eq:p-gradient} give
		\begin{align}\label{eq:delta}
			\delta_t=\beta\left(X_t^+-\frac{\nabla p(z_t)}{2L}\right),\quad
			\norm{\delta_t}^2\le2\beta^2\norm{X_t^+}^2+
			\frac{\beta^2}{2L^2}\norm{\nabla p(z_t)}^2.
		\end{align}
		By
		\eqref{eq:delta} and
		$\ip{\nabla p(z_t)}{X_t^+}\le\norm{\nabla p(z_t)}^2/(8L)+2L\norm{X_t^+}^2$, we have
		\begin{align}
			p(z_{t+1})-p(z_t)
			&\le\ip{\nabla p(z_t)}{\delta_t}+3L\norm{\delta_t}^2
			\le-\frac\beta{4L}\norm{\nabla p(z_t)}^2
			+\frac52\beta L\norm{X_t^+}^2.\label{eq:outer}
		\end{align}
		Set
		$
		\mathcal B_t:=\frac{hs}{8L}\norm{P_t^+}^2
		+\frac{7hs}{4096L}\norm{v_{t+1}^{0}}^2
		+\frac{h\tau s}{128}\norm{Y_t^+}^2
		+\frac{7hLs}{512}\norm{X_t^+}^2,
		$ 
		then we can get
		$0\le\mathcal B_t\le\mathcal D_t^{0}$, by
		Cauchy--Schwarz, we also have
		\begin{align}\label{eq:anchor-cauchy}
			\left[2L\norm{X_t^+}+4\norm{P_t^+}+\frac18\norm{v_{t+1}^{0}}
			+\frac{s}{128}\norm{Y_t^+}\right]^2
			\le\frac{432L}{hs}\mathcal B_t.
		\end{align}
		Combining ~\eqref{eq:moving},
		\eqref{eq:anchor-cauchy}, and $ab\le a^2/4+b^2$, we get
		\begin{align}
			H_{z_{t+1}}(w_{t+1}^{0})-H_{z_t}(w_{t+1}^{0})
			&\le\frac14\mathcal B_t+\frac{704L}{hs}\norm{\delta_t}^2
			\le\frac14\mathcal B_t+\frac{1408L\beta^2}{hs}\norm{X_t^+}^2
			+\frac{352\beta^2}{Lhs}\norm{\nabla p(z_t)}^2.\label{eq:anchor-cost}
		\end{align}
		Substituting~\eqref{eq:virtual-fixed},~\eqref{eq:outer} and~\eqref{eq:anchor-cost}
		into \eqref{lem7:vt}, we get
		\begin{align*}
			\mathcal V(z_{t+1},w_{t+1}^{0})-\mathcal V_t\le
			-\mathcal D_t^{0}+\frac12\mathcal B_t
			-\frac\beta{8L}\norm{\nabla p(z_t)}^2+\frac{17}{L}U_t^2.
		\end{align*}
		Using $\mathcal B_t\le\mathcal D_t^{0}$ yields~\eqref{eq:virtual-joint}.
	\end{proof}
	
	\begin{lemma}
		\label{prop:spde-sample-path}
		Suppose that Assumption~\ref{ass:model} holds and
		Algorithm~\ref{alg:stochastic} uses~\eqref{eq:step-damping}
		and~\eqref{eq:center-relaxation}. For realized batches
		$\Omega_t^{(0)},\Omega_t^{(1)},\Omega_t^{(2)}$, define
		\begin{align}
			\Lambda_t(\Omega_t^{(2)}):={}&
			-\frac{2k}{L}\ip{-\nabla_yG(x_{t+1},y_{t+1};z_{t+1})
				+n_{t+1}-v_{t+1}^{0}}{c_{y,t}}\nonumber\\
			&+\frac{sk}{16L}\ip{c_{y,t}}{y_{t+1}-y^\star(z_{t+1})}
			+\frac{sk}{128L^2}\ip{v_{t+1}^{0}}{c_{y,t}}+\left(\frac1L+\frac{s}{256L^2}\right)k^2\norm{c_{y,t}}^2,
			\label{eq:sample-endpoint-error}
		\end{align}
		where
		$
		c_{y,t}=\frac1m\sum_{i=1}^m
		[\nabla_y f(x_{t+1},y_{t+1};\omega_{t,i}^{(2)})
		-\nabla_yF(x_{t+1},y_{t+1})].
		$
		Then the following inequality holds,
		\begin{align}\label{eq:spde-sample-path-descent}
			\mathcal V_{t+1}-\mathcal V_t
			\le-\frac12\mathcal D_t^{0}
			-\frac{\beta}{8L}\norm{\nabla p(z_t)}^2
			+\frac{17}{L}U_t(\Omega_t^{(0)},\Omega_t^{(1)})^2
			+\Lambda_t(\Omega_t^{(2)}).
		\end{align}
	\end{lemma}
	\begin{proof}
		At center $z_{t+1}$, $w_{t+1}$ and $w_{t+1}^{0}$ differ
		only in $v_{t+1}=v_{t+1}^{0}+kc_{y,t}$ by~\eqref{eq:virtual-noise}.
		Thus 
		\begin{align}\label{eq:noise-expansion}
			\mathcal V_{t+1}-\mathcal V(z_{t+1},w_{t+1}^{0})
			=H_{z_{t+1}}(w_{t+1})-H_{z_{t+1}}(w_{t+1}^{0})
			=\Lambda_t(\Omega_t^{(2)}).
		\end{align}
		Adding~\eqref{eq:virtual-joint} proves
		\eqref{eq:spde-sample-path-descent}.
	\end{proof}
	
	\begin{theorem}\label{thm:drift}
		Under Assumptions~\ref{ass:model} and~\ref{ass:oracle},
		Algorithm~\ref{alg:stochastic} with~\eqref{eq:step-damping}
		and~\eqref{eq:center-relaxation} satisfies
		\begin{align}\label{eq:drift}
			\E[\mathcal V_{t+1}\mid\mathcal F_t]
			\le\mathcal V_t-\frac12\E[\mathcal D_t^{0}\mid\mathcal F_t]
			-\frac\beta{8L}\norm{\nabla p(z_t)}^2
			+\frac{300\sigma^2}{Lm}.
		\end{align}
	\end{theorem}
	\begin{proof}
		Taking the conditional expectation of the pathwise bound~\eqref{eq:spde-sample-path-descent} with respect to $\mathcal F_t$, we note that all coefficients of the terms linear in $c_{y,t}$ in~\eqref{eq:sample-endpoint-error} are $\mathcal F_t^{(2)}$-measurable. Since
		$\E[c_{y,t}\mid\mathcal F_t^{(2)}]=0$
		these terms vanish after conditioning, and
		\begin{align}\label{eq:noise-expectation}
			\begin{aligned}
				&\E[\Lambda_t(\Omega_t^{(2)})\mid\mathcal F_t^{(2)}]=\left(\frac1L+\frac{s}{256L^2}\right)k^2
				\E[\norm{c_{y,t}}^2\mid\mathcal F_t^{(2)}]
				\le\frac{2\sigma^2}{Lm}.
			\end{aligned}
		\end{align}
		For the first two batches,~\eqref{eq:sample-U} yields
		$U_t^2\le8d^2\norm{a_t}^2+8\norm{b_t}^2$. Using
		\eqref{eq:conditional-noise} and $\mathcal F_t\subseteq\mathcal F_t^{(1)}$, we obtain
		\begin{align}
			&\E[\norm{b_t}^2\mid\mathcal F_t]
			=\E\!\left[\E[\norm{b_t}^2\mid\mathcal F_t^{(1)}]
			\,\middle|\,\mathcal F_t\right]\le\frac{\sigma^2}{m},\quad
			\frac{17}{L}\E[U_t^2\mid\mathcal F_t]
			\le\frac{136(1+d^2)\sigma^2}{Lm}.\label{eq:U-expectation}
		\end{align}
		Taking the $\mathcal F_t$-conditional expectation
		in~\eqref{eq:spde-sample-path-descent} and substituting
		\eqref{eq:noise-expectation} and~\eqref{eq:U-expectation}, we obtain
		\begin{align*}
			\E[\mathcal V_{t+1}-\mathcal V_t\mid\mathcal F_t]
			\le{}&-\frac12\E[\mathcal D_t^{0}\mid\mathcal F_t]
			-\frac\beta{8L}\norm{\nabla p(z_t)}^2
			+\frac{[136(1+d^2)+2]\sigma^2}{Lm}.
		\end{align*}
		Since $136(1+d^2)+2\le300$, this yields~\eqref{eq:drift}.
	\end{proof}
	\subsection{Stationarity and complexity results}
	To sum the descent inequality~\eqref{eq:drift}, we first
	bound the initial potential.
	\begin{lemma}\label{lem:initial}
		Suppose that Assumption~\ref{ass:model} holds and $0<\tau\le L$. Then the
		initialization in Algorithm~\ref{alg:stochastic} satisfies $0\le\mathcal V_0\le B$, where 
		$
		B:=\phi(x_0)-\phi_{\inf}+2D_{\mathcal Y}\norm{\nabla_y F(x_0,y_0)}+\frac{\norm{\nabla_x F(x_0,y_0)}^2+2\norm{\nabla_y F(x_0,y_0)}^2}{L}+\left(3+\frac1{64}\right)LD_{\mathcal Y}^2.
		$
	\end{lemma}
	\begin{proof}
		Since $z_0=x_0$ and $v_0=0$,
		$\mathcal V_0=p(x_0)-\inf_zp(z)+E_{x_0}(w_0)$.
		Moreover,
		$p(x_0)\le\phi_\tau(x_0)\le\phi(x_0)$ and
		$\inf_zp(z)=\inf_x\phi_\tau(x)\ge\phi_{\inf}-\tau D_{\mathcal Y}^2/2$.
		Therefore,
		$
		p(x_0)-\inf_zp(z)\le\phi(x_0)-\phi_{\inf}+\frac\tau2D_{\mathcal Y}^2.
		$
		Substitute $\xi_0=n_0=v_0=0$ into~\eqref{eq:energy}, we have
		\begin{align*}
			E_{x_0}(w_0)={}&\left(1-\frac{s}{16L}\right)
			\left[p(x_0)-F(x_0,y_0)+\frac\tau2\norm{y_0}^2\right]\\
			&+\frac{\norm{\nabla_x F(x_0,y_0)}^2+
				\norm{-\nabla_y F(x_0,y_0)+\tau y_0}^2}{L}
			+\frac\tau{256}\norm{y_0-y^\star(x_0)}^2.
		\end{align*}
		Using $p(x_0)\le\phi(x_0)$ and the concavity of
		$F(x_0,\cdot)$, we obtain
		\begin{align*}
			p(x_0)-F(x_0,y_0)+\frac\tau2\norm{y_0}^2
			&\le\phi(x_0)-F(x_0,y_0)+\frac\tau2D_{\mathcal Y}^2
			\le2D_{\mathcal Y}\norm{\nabla_yF(x_0,y_0)}
			+\frac\tau2D_{\mathcal Y}^2.
		\end{align*}		
		Combining the two bounds, using
		$\norm{y_0},\norm{y^\star(x_0)}\le D_{\mathcal Y}$, $\tau\le L$, $s<16L$ and Lemma~\ref{lem:positive}, we obtain $0\le\mathcal V_0\le B$.
	\end{proof}
	
	\begin{lemma}\label{lem:budget}
		Suppose that Assumptions~\ref{ass:model} and~\ref{ass:oracle} hold, and let
		SPDE use~\eqref{eq:step-damping} and~\eqref{eq:center-relaxation}.
		Then, 
		\begin{align}
			\sum_{t=0}^{T-1}\left[\frac12\E\mathcal D_t^{0}
			+\frac\beta{8L}\E\norm{\nabla p(z_t)}^2\right]\le B+\frac{300T\sigma^2}{Lm}.\label{eq:sum-D}
		\end{align}
		The random output satisfies
		\begin{align}\label{eq:R-bound}
			\E\mathcal R(x_{J+1},y_{J+1})^2
			\le\frac{64LB}{\beta T}+\frac{19200\sigma^2}{\beta m}
			+2\tau^2D_{\mathcal Y}^2.
		\end{align}
	\end{lemma}
	\begin{proof}
		First take total expectation in~\eqref{eq:drift}, then sum
		from $t=0$ to $T-1$, we have
		\begin{align*}
			\sum_{t<T}\left[\frac12\E\mathcal D_t^0
			+\frac\beta{8L}\E\norm{\nabla p(z_t)}^2\right]
			&\le\mathcal V_0-\E\mathcal V_T+\frac{300T\sigma^2}{Lm}
			\le B+\frac{300T\sigma^2}{Lm}.
		\end{align*}
		Define $M_T:=B+\frac{300T\sigma^2}{Lm}$. Substituting~\eqref{fc:dissipation} into~\eqref{eq:sum-D}
		and using~\eqref{eq:center-relaxation}, we obtain
		\begin{align}\label{eq:component-budgets}
			&\sum_{t<T}\E\norm{\nabla_xG(x_{t+1},y_{t+1};z_t)+\xi_{t+1}}^2\le\frac{8M_T}{h},\quad \sum_{t<T}\E\norm{v_{t+1}^{0}}^2\le\frac{2LM_T}{7\beta},\nonumber\\
			&\sum_{t<T}\E\norm{-\nabla_yG(x_{t+1},y_{t+1};z_t)+n_{t+1}-v_{t+1}^{0}}^2\le\frac{LM_T}{12\beta},\quad
			\sum_{t<T}\E\norm{x_{t+1}-x^\star(z_t)}^2\le\frac{M_T}{28\beta L},\nonumber\\
			&
			\sum_{t<T}\E\norm{\nabla p(z_t)}^2\le\frac{8LM_T}{\beta}.
		\end{align}
		Define
		$
		S_t:=\norm{\nabla_x F(x_{t+1},y_{t+1})+\xi_{t+1}}^2
		+\norm{-\nabla_y F(x_{t+1},y_{t+1})+\tau y_{t+1}+n_{t+1}}^2.
		$
		By~\eqref{eq:p-gradient}, we have
		\begin{align*}
			\nabla_x F(x_{t+1},y_{t+1})+\xi_{t+1}&=P_t^+-2LX_t^++\nabla p(z_t),\\
			-\nabla_y F(x_{t+1},y_{t+1})+\tau y_{t+1}+n_{t+1}
			&=Q_t^++v_{t+1}^{0}.
		\end{align*}
		Hence
		$
		S_t\le3\norm{P_t^+}^2+12L^2\norm{X_t^+}^2
		+3\norm{\nabla p(z_t)}^2+2\norm{Q_t^+}^2+2\norm{v_{t+1}^{0}}^2.
		$
		Summing the above bounds and using~\eqref{eq:component-budgets} and
		$\beta/(hL)=s/(4096L)\le1/2048$, we get
		$
		\sum_{t<T}\E S_t\le\frac{32LM_T}{\beta}.
		$
		Moreover, by~\eqref{eq:normal-membership},
		$\mathcal R(x_{t+1},y_{t+1})\le\sqrt{S_t}+\tau D_{\mathcal Y}$.
		Squaring this inequality and averaging over the independent uniform index $J$ yields
		\begin{align*}
			\E\mathcal R(x_{J+1},y_{J+1})^2
			=\frac1T\sum_{t<T}\E\mathcal R(x_{t+1},y_{t+1})^2
			\le\frac{64LM_T}{\beta T}+2\tau^2D_{\mathcal Y}^2.
		\end{align*}
		Substituting~$M_T$ proves~\eqref{eq:R-bound}.
	\end{proof}
	
	\subsubsection{Nonconvex--strongly concave setting}\label{sec:plain-ncsc}
	We specialize the mini-batch analysis to strongly concave inner objectives.
	\begin{assumption}\label{ass:strong-concavity}
		There exists $\mu>0$ such that, for all $x\in\mathcal X$
		and $y,y'\in\mathcal Y$,
		\begin{align}\label{eq:strong-concavity}
			F(x,y')\le F(x,y)
			+\ip{\nabla_y F(x,y)}{y'-y}
			-\frac{\mu}{2}\norm{y'-y}^2.
		\end{align}
		Let $\kappa:=L/\mu$. By Assumption~\ref{ass:model},
		$0<\mu\le L$, so $\kappa\ge1$.
	\end{assumption}
	Throughout this subsection, we set $\tau=0$ in $G$ and $p$.
	For fixed $z$, $G(\cdot,\cdot;z)$ is $L$-strongly convex in $x$,
	$\mu$-strongly concave in $y$, and has a $3L$-Lipschitz gradient.
	
	\begin{lemma}\label{lem:sc-plain-residual}
		Suppose that Assumptions~\ref{ass:model},~\ref{ass:oracle}, and
		\ref{ass:strong-concavity} hold. Set
		\begin{align}\label{eq:sc-parameters}
			\begin{gathered}
				L_{G,{\rm sc}}:=3L,\quad s_{\rm sc}:=\sqrt{2L\mu},\quad
				h_{\rm sc}:=\frac1{64L_{G,{\rm sc}}}=\frac1{192L},\\
				\alpha_{\rm sc}:=\left(1+\frac{h_{\rm sc}s_{\rm sc}}{16}\right)^{-1},
				\quad
				k_{\rm sc}:=\frac{h_{\rm sc}(8L-s_{\rm sc})}{16+h_{\rm sc}s_{\rm sc}},\quad
				\beta_{\rm sc}:=\frac{h_{\rm sc}s_{\rm sc}}{4096}
				=\frac{\sqrt2}{786432}\kappa^{-1/2}.
			\end{gathered}
		\end{align}
		Run SPDE with~\eqref{eq:sc-parameters} and $\tau=0$ in all
		updates. Then
		\begin{align}\label{eq:sc-plain-residual}
			\E\mathcal R(x_{J+1},y_{J+1})^2
			\le \frac{64LB}{\beta_{\rm sc}T}
			+\frac{19200\sigma^2}{\beta_{\rm sc}m}.
		\end{align}
	\end{lemma}
	\begin{proof}
		Define $\mathcal V^{\rm sc}$ by~\eqref{eq:energy} using
		$G,p$ with $\tau=0$, $s=s_{\rm sc}$, and the dual distance term
		$\mu\norm{y-y^\star(z)}^2/256$.
		Using $\mu$ as the dual strong-concavity modulus and the parameters
		in~\eqref{eq:sc-parameters}, the proof of
		Theorem~\ref{thm:drift} yields the analogue of~\eqref{eq:drift}
		with the same constants.		
		Moreover, $z_0=x_0$ gives
		\begin{align*}
			\mathcal V_0^{\rm sc}
			&\le \phi(x_0)-\phi_{\inf}
			+2D_{\mathcal Y}\norm{\nabla_y F(x_0,y_0)}
			+\frac{\norm{\nabla_x F(x_0,y_0)}^2
				+\norm{\nabla_y F(x_0,y_0)}^2}{L}
			+\frac{\mu D_{\mathcal Y}^2}{64}\le B.
		\end{align*}
		Under the same substitutions, the proof of
		Lemma~\ref{lem:budget} yields~\eqref{eq:sc-plain-residual}.
	\end{proof}
	
	\begin{theorem}
		\label{thm:sc-plain-complexity}
		Suppose that Assumptions~\ref{ass:model},~\ref{ass:oracle}, and
		\ref{ass:strong-concavity} hold. Fix a known bound $\overline B\ge B$ and set
		\begin{align}\label{eq:sc-plain-budgets}
			T=\max\left\{1,\left\lceil
			\frac{128L\overline B}{\beta_{\rm sc}\eps^2}\right\rceil\right\},
			\qquad
			m=\max\left\{1,\left\lceil
			\frac{38400\sigma^2}{\beta_{\rm sc}\eps^2}\right\rceil\right\}.
		\end{align}
		Then the output of Algorithm~\ref{alg:stochastic} satisfies~\eqref{eq:goal}, and the number $N_{\rm sc}$ of SFO calls satisfies
		\begin{align}
			N_{\rm sc}
			&\le3\left(1+\frac{128L\overline B}{\beta_{\rm sc}\eps^2}\right)
			\left(1+\frac{38400\sigma^2}{\beta_{\rm sc}\eps^2}\right)
			=O\!\left(1+
			\frac{(L\overline B+\sigma^2)\sqrt\kappa}{\eps^2}
			+\frac{L\overline B\sigma^2\kappa}{\eps^4}\right),
			\label{eq:sc-plain-oracle}
		\end{align}
		the iteration count is $T=O(\sqrt\kappa\eps^{-2})$ and the batch size is
		$m=O(\sqrt\kappa\eps^{-2})$.
	\end{theorem}
	\begin{proof}
		By~\eqref{eq:sc-plain-budgets}, each term on the right-hand side of~
		\eqref{eq:sc-plain-residual} is bounded by $\eps^2/2$, which proves~\eqref{eq:goal}. Since each of the $T$ iterations uses three mini-batches of size $m$, the total number of SFO calls is $N_{\rm sc}=3mT$.
		Finally, using $\max\{1,\lceil a\rceil\}\le1+a$ for $a\ge0$ and
		$\beta_{\rm sc}^{-1}=(786432/\sqrt2)\sqrt\kappa$, yields
		\eqref{eq:sc-plain-oracle}.
	\end{proof}
	
	\begin{remark}
		Theorem~\ref{thm:sc-plain-complexity} gives an
		$O(\kappa\eps^{-4})$ SFO complexity for SPDE to find an
		$\eps$-game-stationary point in the stochastic
		nonconvex--strongly concave setting, where $\kappa=L/\mu$.
	\end{remark}
	
	\subsubsection{Nonconvex--concave setting}\label{sec:plain-ncc}
	Under concavity, choosing $\tau=\Theta(\eps)$ for small $\eps$
	controls the regularization bias and gives $\beta=\Theta(\eps^{1/2})$.
	
	\begin{theorem}\label{thm:complexity}
		Suppose that Assumptions~\ref{ass:model} and~\ref{ass:oracle} hold.
		Fix a known $\overline B\ge B$ and $\eps>0$, set
		\begin{align}\label{eq:tau-choice}
			\tau=\min\left\{L,\frac{\eps}{2D_{\mathcal Y}}\right\},
		\end{align}
		choose $h,\alpha,k$ by~\eqref{eq:step-damping} and $\beta$
		by~\eqref{eq:center-relaxation}, and set
		\begin{align}\label{eq:budgets}
			T=\max\left\{1,\left\lceil\frac{256L\overline B}{\beta\eps^2}\right\rceil\right\},
			\qquad
			m=\max\left\{1,\left\lceil\frac{76800\sigma^2}{\beta\eps^2}\right\rceil\right\}.
		\end{align}
		Then the output of Algorithm~\ref{alg:stochastic} satisfies~\eqref{eq:goal}, and
		the number $N$ of stochastic full-gradient oracle calls satisfies
		\begin{align}\label{eq:exact-budget}
			N\le3mT\le3\left(1+\frac{256L\overline B}{\beta\eps^2}\right)
			\left(1+\frac{76800\sigma^2}{\beta\eps^2}\right)=O\!\left(
			1+\frac{(L\overline B+\sigma^2)K_\eps}{\eps^2}
			+\frac{L\overline B\sigma^2K_\eps^2}{\eps^4}\right),
		\end{align}
		with $K_\eps:=\max\{1,\sqrt{2LD_{\mathcal Y}/\eps}\}$,
		the iteration count is $T=O(\eps^{-5/2})$ and the batch size is
		$m=O(\eps^{-5/2})$.
	\end{theorem}
	\begin{proof}
		By~\eqref{eq:tau-choice} and~\eqref{eq:budgets},
		the three terms on the right-hand side of~\eqref{eq:R-bound} are bounded by $\eps^2/4$, $\eps^2/4$, and $\eps^2/2$, respectively, which proves~\eqref{eq:goal}.
		Applying $\max\{1,\lceil a\rceil\}\le1+a$ for $a\ge0$ gives
		the first two inequalities in~\eqref{eq:exact-budget}.
		By~\eqref{eq:step-damping},~\eqref{eq:center-relaxation},
		and $0<\tau\le L$,
		$
		\beta^{-1}=\frac{4096}{hs}
		=\Theta\!\left(\sqrt{L/\tau}\right)=\Theta(K_\eps).
		$
		Substitution into the product bound gives the asymptotic
		estimate in~\eqref{eq:exact-budget}.
	\end{proof}
	\begin{remark}
		Theorem~\ref{thm:complexity} gives an $O(\eps^{-5})$ SFO
		complexity for SPDE to find an $\eps$-game-stationary point
		in the stochastic nonconvex--concave setting.
	\end{remark}
	
	\section{A variance-reduced SPDE method}\label{sec:vr}
	We introduce VR-SPDE, a variance-reduced variant that preserves
	SPDE's single-loop structure.
	A single recursive gradient estimator tracks the current,
	predictor, and corrected query points, replacing the three
	independent mini-batch estimates in Algorithm~\ref{alg:stochastic}.
	All other updates remain unchanged, and periodic estimator
	refreshes require no inner optimization.
	
	Unlike fresh mini-batch errors, recursive endpoint errors are
	generally not conditionally mean-zero, so the linear noise terms
	in~\eqref{eq:noise-expansion} no longer cancel.
	We control the endpoint error pathwise, bound the accumulated
	estimator variance by query-point motion, and absorb these bounds
	into the descent inequality.
	This yields the variance-reduced GS complexity guarantees below;
	the OS guarantees are established in Section~\ref{sec:os}.

	\subsection{Paired oracle and recursive estimator}
	
	Variance reduction requires access to stochastic gradient
	differences evaluated with a common sample.
	We impose the following additional oracle assumption.
	
	\begin{assumption}\label{ass:paired}
		In addition to Assumption~\ref{ass:oracle}, the oracle permits
		evaluating $\nabla f(u;\omega)$ and $\nabla f(u';\omega)$
		using the same sample $\omega$.
		There exists $\ell\ge L$ such that, for every pre-query
		sigma-algebra $\mathcal A$ and every pair of
		$\mathcal A$-measurable feasible points
		$u,u'\in \mathcal X\times \mathcal Y$, a fresh sample
		$\omega\sim\mathcal D$, independent of $\mathcal A$,
		satisfies
		\begin{align}\label{eq:paired-smoothness}
			\E\!\left[
			\norm{\nabla f(u;\omega)-\nabla f(u';\omega)}^2
			\,\middle|\,\mathcal A
			\right]
			\le\ell^2\norm{u-u'}^2.
		\end{align}
		Samples within each newly drawn batch are independent,
		and each batch is independent of the information available
		before it is drawn.
		Each sample-gradient evaluation at one point counts as one
		SFO call.
	\end{assumption}
	
	Let $\chi:=\ell/L\ge1$.
	Assumption~\ref{ass:paired} bounds stochastic gradient differences
	in mean square. It is stronger than Lipschitz continuity of the population gradient
	and is required only for VR-SPDE.
	
	The estimator follows the SPIDER
	principle~\cite{fang2018spider}.
	Our analysis accounts for adaptive query points, each of which may
	depend on estimates computed at earlier queries.
	
	\paragraph{A recursive estimator along the query sequence.}
	Index the estimator requests by $j=0,\ldots,3T-1$, with query points
	\begin{align}\label{eq:query-stream}
		u_{3t}&=(x_t,y_t),\qquad
		u_{3t+1}=(\widetilde x_t,\widetilde y_t),\qquad
		u_{3t+2}=(x_{t+1},y_{t+1}),
		\qquad 0\le t<T.
	\end{align}
	The algorithm generates these points sequentially.
	For integers $q,B_{\rm r},b\ge1$, define
	\begin{align}\label{eq:spider-estimator}
		\mathsf g_j=
		\begin{cases}
			\displaystyle
			\frac{1}{B_{\rm r}}
			\sum_{i=1}^{B_{\rm r}}\nabla f(u_j;\omega_{j,i}),
			& j\equiv0\pmod q,\\[7pt]
			\displaystyle
			\mathsf g_{j-1}
			+\frac{1}{b}\sum_{i=1}^{b}
			\bigl[
			\nabla f(u_j;\omega_{j,i})
			-\nabla f(u_{j-1};\omega_{j,i})
			\bigr],
			& j\not\equiv0\pmod q.
		\end{cases}
	\end{align}
	At each request, a fresh batch is drawn after the query points
	have been determined, and the two gradients in each paired
	difference use the same sample.
	Here $B_{\rm r}$ and $b$ are the refresh and difference batch sizes,
	respectively, and $q$ is the refresh period counted in estimator
	requests.
	In particular, $\mathsf g_0$ is always a refresh estimate.
	
	At iteration boundaries, the consecutive query points satisfy
	\[
	u_{3t}=u_{3t-1}=(x_t,y_t),\qquad 1\le t<T.
	\]
	If request $3t$ is not a refresh, all paired differences vanish,
	so we set $\mathsf g_{3t}=\mathsf g_{3t-1}$ without additional
	oracle calls.
	Otherwise, we draw the prescribed refresh batch.

	\begin{algorithm}[!ht]
		\caption{Variance-reduced stochastic projected damped
			extragradient (VR-SPDE)}
		\label{alg:vr}
		\begin{algorithmic}[1]
			\Require
			Deterministic $x_0\in \mathcal X$ and $y_0\in \mathcal Y$;
			integers $T,q,B_{\rm r},b\ge1$;
			$0\le\tau\le L$, $h>0$, $0<\alpha\le1$,
			$k\ge0$, and $0<\beta\le1$.
			\State
			Set $z_0=x_0$, $\xi_0=0\in\R^n$,
			and $n_0=v_0=0\in\R^p$.
			\For{$t=0,\ldots,T-1$}
			\State
			Set $u_{3t}=(x_t,y_t)$ and compute
			$\mathsf g_{3t}$ by~\eqref{eq:spider-estimator}.
			\State
			Set $\widehat g_t^{(0)}=\mathsf g_{3t}$ and
			compute $(\widetilde x_t,\widetilde y_t)$
			using the predictor updates in
			Algorithm~\ref{alg:stochastic}.
			
			\State
			Set $u_{3t+1}=(\widetilde x_t,\widetilde y_t)$
			and compute $\mathsf g_{3t+1}$
			by~\eqref{eq:spider-estimator}.
			\State
			Set $\widehat g_t^{(1)}=\mathsf g_{3t+1}$ and
			compute $\bar v_t,x_{t+1},\xi_{t+1},
			y_{t+1},n_{t+1}$ using the correction updates
			in Algorithm~\ref{alg:stochastic}.
			
			\State
			Set $u_{3t+2}=(x_{t+1},y_{t+1})$ and compute
			$\mathsf g_{3t+2}$ by~\eqref{eq:spider-estimator}.
			\State
			Set $\widehat g_t^{(2)}=\mathsf g_{3t+2}$ and
			compute $v_{t+1}$ and $z_{t+1}$ using the
			momentum and center updates in
			Algorithm~\ref{alg:stochastic}.
			\EndFor
			\Ensure
			Draw $J\sim\operatorname{Unif}\{0,\ldots,T-1\}$
			and set $(x_{\rm out},y_{\rm out})
			=(x_{J+1},y_{J+1})$.
		\end{algorithmic}
	\end{algorithm}
	
	\paragraph{The VR-SPDE algorithm.}
	Algorithm~\ref{alg:vr} replaces SPDE's three independent mini-batch
	estimates with one recursive estimator along the query sequence.
	The parameters $q,B_{\rm r},b$ control estimator refresh and
	accuracy; $\tau,h,\alpha,k,\beta$ retain their roles in
	Algorithm~\ref{alg:stochastic}.
	Their choices are specified in the complexity results below.
	
	Like VR-SAPD+~\cite{zhang2022sapdplus}, VR-SPDE uses periodic
	refreshes and paired gradient differences, whereas
	Acc-MDA~\cite{huang2022accelerated} and
	FORMDA~\cite{zhangxu2025formda} use STORM-type recursions
	without periodic refreshes.
	VR-SPDE combines the SPIDER estimator~\cite{fang2018spider}
	with SPDE's prediction, correction, damped dual momentum,
	and center updates.
	Under Assumption~\ref{ass:paired}, the variance of each gradient
	difference is controlled by the distance between query points.
	Small query-point motion therefore allows small difference
	batches to maintain accuracy, yielding the complexity
	improvements over mini-batch SPDE established below.
	
	Every occurrence of the predictor estimate uses the same vector
	$\mathsf g_{3t+1}$.
	Refreshes update only the estimator, preserving the single-loop
	structure.
	The displayed output rule is used for GS; the OS output is
	specified in Section~\ref{sec:os}.
	There are $N_{\rm r} :=1+\left\lfloor\frac{3T-1}{q}\right\rfloor =\left\lceil\frac{3T}{q}\right\rceil$
	refresh requests. Counting $B_{\rm r}$ SFO calls per refresh and at most $2b$
	per nonrefresh request gives $N_{\rm SFO} \le B_{\rm r}N_{\rm r}+2b(3T-N_{\rm r})$.
	
	\paragraph{Conditional estimator errors.}
	Let $\mathcal H_j$ denote the sigma-algebra containing all
	information available immediately before request $j$ is
	processed.
	The query point $u_j$ is $\mathcal H_j$-measurable.
	For $j\ge1$, the preceding point $u_{j-1}$ and estimate
	$\mathsf g_{j-1}$ are also $\mathcal H_j$-measurable.
	Define
	\begin{align}\label{eq:vr-errors}
		\eta_j:=\mathsf g_j-\nabla F(u_j),
		\qquad
		a_t:=\eta_{3t},
		\qquad
		b_t:=\eta_{3t+1},
		\qquad
		c_t:=\eta_{3t+2}.
	\end{align}
	These errors satisfy the algebraic decompositions
	in~\eqref{eq:sample-batch-identities}.
	Their conditional moments, however, differ from those of
	independent mini-batch estimates.
	
	At a refresh request,
	Assumption~\ref{ass:oracle} gives
	$
	\E[\eta_j\mid\mathcal H_j]=0,~
	\E[\norm{\eta_j}^2\mid\mathcal H_j]
	\le\frac{\sigma^2}{B_{\rm r}}.
	$
	At a nonrefresh request, denote
	$
	\eta_j=\eta_{j-1}+\delta_j,
	$
	where
	$
	\delta_j
	:=
	\frac{1}{b}\sum_{i=1}^{b}
	\bigl[
	\nabla f(u_j;\omega_{j,i})
	-\nabla f(u_{j-1};\omega_{j,i})
	\bigr]
	-\bigl[\nabla F(u_j)-\nabla F(u_{j-1})\bigr].
	$
	Conditional unbiasedness of the paired differences and
	Assumption~\ref{ass:paired} imply
	\[
	\E[\delta_j\mid\mathcal H_j]=0,
	\qquad
	\E[\norm{\delta_j}^2\mid\mathcal H_j]
	\le\frac{\ell^2}{b}\norm{u_j-u_{j-1}}^2.
	\]
	Since $\eta_{j-1}$ is $\mathcal H_j$-measurable, it follows that
	\[
	\E[\eta_j\mid\mathcal H_j]=\eta_{j-1},
	\qquad
	\E[\norm{\eta_j}^2\mid\mathcal H_j]
	\le
	\norm{\eta_{j-1}}^2
	+\frac{\ell^2}{b}\norm{u_j-u_{j-1}}^2.
	\]
	
	\paragraph{Feasibility and moment bounds.}
	Lemma~\ref{lem:feasibility} applies to VR-SPDE.
	Thus, almost surely,
	\begin{align*}
		x_t,z_t\in\mathcal X,\qquad
		y_t\in\mathcal Y,\qquad
		\xi_t\in N_{\mathcal X}(x_t),\qquad
		n_t\in N_{\mathcal Y}(y_t),
	\end{align*}
	and all predictor points lie in $\mathcal X\times\mathcal Y$.
	All states and gradient estimates have finite second moments
	at every finite iteration.
	
	We combine a pathwise descent bound with bounds on query-point
	motion and accumulated estimator error.
	Choosing the refresh period and batch sizes appropriately
	allows us to absorb the error terms into the descent bound
	and derive complexity guarantees for both the strongly concave
	and merely concave settings.
	
	\subsection{Convergence analysis}
	
	Retain the virtual momentum and state in~\eqref{eq:virtual}, the quantities
	$\mathcal I_t^0,\mathcal D_t^0$ from Lemma~\ref{lem:virtual-error}, and
	$P_t^+,Q_t^+,X_t^+$ defined in the proof of Lemma~\ref{lem:budget}.
	Also write $Y_t^+=y_{t+1}-y^\star(z_t)$.
	The deterministic error bound~\eqref{eq:error-U} and virtual joint
	descent~\eqref{eq:virtual-joint} hold for these correlated errors: their
	proofs use only the update identities and inequalities that hold pathwise.
	After inserting~\eqref{eq:spider-estimator}, the errors $a_t,b_t,c_t$ are
	explicit functions of the realized samples $\{\omega_{j,i}\}$.  The next
	lemma is therefore proved for a fixed realization of the complete recursive
	estimator path; expectations enter only in
	Lemma~\ref{lem:vr-error} and the subsequent accumulated bounds.
	
	\begin{lemma}\label{lem:vr-pathwise}
		Suppose that Assumption~\ref{ass:model} holds, and let the iterates be
		generated by Algorithm~\ref{alg:vr} with the parameter choices
		in~\eqref{eq:step-damping} and~\eqref{eq:center-relaxation}.
		Set $C_\star:=140$. For every realization of the samples
		$\{\omega_{j,i}\}_{0\le j<3T}$ used in~\eqref{eq:spider-estimator}, we have
		\begin{align}\label{eq:vr-pathwise}
			\mathcal V_{t+1}-\mathcal V_t
			\le-\frac14\mathcal D_t^0
			-\frac\beta{16L}\norm{\nabla p(z_t)}^2
			+\frac{C_\star}{L\beta}
			(\norm{a_t}^2+\norm{b_t}^2+\norm{c_t}^2).
		\end{align}
	\end{lemma}
	\begin{proof}
		Fix a realization of the samples. We bound the difference between the actual and virtual potentials in~\eqref{eq:noise-expansion}. Let
		$\delta_t=z_{t+1}-z_t$. By~\eqref{eq:sensitivity},
		\eqref{eq:delta}, and $s^2=2L\tau$, we have
		\begin{align}
			\norm{y_{t+1}-y^\star(z_{t+1})}
			&\le\norm{Y_t^+}+\sqrt{\frac{2L}{\tau}}\norm{\delta_t},\nonumber\\
			\frac{sk}{16L}\sqrt{\frac{2L}{\tau}}\norm{\delta_t}\norm{c_t}
			&\le\frac{k\beta}{8}\norm{X_t^+}\norm{c_t}
			+\frac{k\beta}{16L}\norm{\nabla p(z_t)}\norm{c_t}.
			\label{eq:vr-shifted-cross}
		\end{align}
		By~\eqref{fc:dissipation} and $\beta=hs/4096$, we get
		\begin{align}\label{eq:vr-D-components}
			\mathcal D_t^0\ge
			\frac{24\beta}{L}\norm{Q_t^+}^2
			+\frac{7\beta}{L}\norm{v_{t+1}^0}^2
			+32\tau\beta\norm{Y_t^+}^2
			+56L\beta\norm{X_t^+}^2.
		\end{align}
		Since $\nabla_y G$ does not depend on the center,
		applying the Cauchy--Schwarz inequality to
		\eqref{eq:sample-endpoint-error} and using
		$\norm{c_{y,t}}\le\norm{c_t}$ and~\eqref{eq:vr-shifted-cross}
		in~\eqref{eq:noise-expansion}, we obtain
		\begin{align*}
			&\mathcal V_{t+1}-\mathcal V(z_{t+1},w_{t+1}^0)\\
			&\le\left[\frac{2k}{L}\norm{Q_t^+}
			+\frac{sk}{16L}\norm{Y_t^+}
			+\frac{sk}{128L^2}\norm{v_{t+1}^0}
			+\frac{k\beta}{8}\norm{X_t^+}
			+\frac{k\beta}{16L}\norm{\nabla p(z_t)}\right]\norm{c_t}\\
			&\quad+\left(\frac1L+\frac{s}{256L^2}\right)k^2\norm{c_t}^2.
		\end{align*}
		Applying Young's inequality to the five cross terms and using
		\eqref{eq:vr-D-components}, $s^2=2L\tau$, $s\le2L$, and
		$\beta\le1$, we obtain
		\[
		\mathcal V_{t+1}-\mathcal V(z_{t+1},w_{t+1}^0)
		\le\frac1{16}\mathcal D_t^0
		+\frac\beta{16L}\norm{\nabla p(z_t)}^2
		+\frac{7k^2}{L\beta}\norm{c_t}^2.
		\]
		Adding this bound to~\eqref{eq:virtual-joint}  leaves
		$7\mathcal D_t^0/16\ge\mathcal D_t^0/4$ and
		$\beta\norm{\nabla p(z_t)}^2/(16L)$ as descent terms.
		Finally, using $U_t^2\le8d^2\norm{a_t}^2+8\norm{b_t}^2$, we obtain
		$
		\frac{17}{L}U_t^2+\frac{7k^2}{L\beta}\norm{c_t}^2
		\le\frac{C_\star}{L\beta}
		\bigl(\norm{a_t}^2+\norm{b_t}^2+\norm{c_t}^2\bigr),
		$
		which proves~\eqref{eq:vr-pathwise}.
	\end{proof}

	\begin{lemma}\label{lem:vr-motion}
		Suppose that Assumption~\ref{ass:model} holds, and let the iterates be
		generated by Algorithm~\ref{alg:vr} with the parameter choices
		in~\eqref{eq:step-damping} and~\eqref{eq:center-relaxation}. For each iteration, the query points in~\eqref{eq:query-stream} satisfy
		\begin{align}\label{eq:vr-motion}
			\norm{u_{3t+1}-u_{3t}}^2+\norm{u_{3t+2}-u_{3t+1}}^2
			\le\frac{\mathcal D_t^0}{L\beta}
			+\frac{\norm{a_t}^2+\norm{b_t}^2}{L^2}.
		\end{align}
	\end{lemma}
	\begin{proof}
		We bound the two query displacements in terms of the fixed-center residuals. Let $P$ and $Q$ denote the residuals associated with $w_t$ in Lemma~\ref{lem:positive}. By projection nonexpansiveness and the definition $\mathcal I_t^0=\norm{P-P_t^+}^2+\norm{Q-Q_t^+}^2$, we obtain
		\begin{align*}
			\norm{u_{3t+1}-u_{3t}}
			&\le h\left(\sqrt{\norm{P_t^+}^2+\norm{Q_t^+}^2}
			+\sqrt{\mathcal I_t^0}+\norm{a_t}\right),\\
			\norm{u_{3t+2}-u_{3t+1}}
			&\le h\left(\sqrt{\mathcal I_t^0}+\norm{e_t}+\norm{a_t}\right).
		\end{align*}
		By~\eqref{eq:error-U} and use the definition of $U_t$, we obtain
		$
		\norm{e_t}^2\le8d^2\mathcal I_t^0+2U_t^2,
		~ U_t^2\le8d^2\norm{a_t}^2+8\norm{b_t}^2.
		$
		Substituting these estimates gives
		\begin{align}
			&\norm{u_{3t+1}-u_{3t}}^2+\norm{u_{3t+2}-u_{3t+1}}^2\le8h^2\bigl(\norm{P_t^+}^2+\norm{Q_t^+}^2+\mathcal I_t^0\bigr)
			+64h^2(\norm{a_t}^2+\norm{b_t}^2).
			\label{eq:vr-motion-intermediate}
		\end{align}
		The corresponding terms in~\eqref{fc:dissipation} satisfy
		$
		\norm{P_t^+}^2\le\frac4h\mathcal D_t^0,~
		\norm{Q_t^+}^2\le\frac{L}{24\beta}\mathcal D_t^0,~
		\mathcal I_t^0\le4L\mathcal D_t^0.
		$
		Substituting these bounds into~\eqref{eq:vr-motion-intermediate}, using
		$hL\le1/64$, $\beta=hs/4096$ and $s\le2L$, proves~\eqref{eq:vr-motion}.
	\end{proof}
	
	Define the cumulative expected error and dissipation by
	\begin{align}\label{eq:vr-cumulative-definitions}
		\mathcal E_T:=\sum_{j=0}^{3T-1}\E\norm{\eta_j}^2
		=\sum_{t=0}^{T-1}\E(\norm{a_t}^2+\norm{b_t}^2+\norm{c_t}^2),
		\qquad
		\mathcal A_T:=\sum_{t=0}^{T-1}\E\mathcal D_t^0.
	\end{align}
	
	\begin{lemma}\label{lem:vr-error}
		Suppose that Assumptions~\ref{ass:model},~\ref{ass:oracle}, and
		\ref{ass:paired} hold, and let the iterates and estimators be generated by
		Algorithm~\ref{alg:vr} with the parameter choices in~\eqref{eq:step-damping}
		and~\eqref{eq:center-relaxation}. Then
		\begin{align}\label{eq:vr-cumulative-error}
			\mathcal E_T\le\frac{3T\sigma^2}{B_{\rm r}}
			+\frac{q\chi^2L}{b\beta}\mathcal A_T
			+\frac{q\chi^2}{b}\mathcal E_T.
		\end{align}
	\end{lemma}
	\begin{proof}
		At a refresh index,~\eqref{eq:spider-estimator} and
		Assumption~\ref{ass:oracle} imply
		$\E[\norm{\eta_j}^2\mid\mathcal H_j]\le\sigma^2/B_{\rm r}$.
		At a nonrefresh index, define the centered difference error
		\[
		\zeta_j:=\frac1b\sum_{i=1}^b
		[\nabla f(u_j;\omega_{j,i})-\nabla f(u_{j-1};\omega_{j,i})]
		-[\nabla F(u_j)-\nabla F(u_{j-1})].
		\]
		Subtracting $\nabla F(u_j)$ from~\eqref{eq:spider-estimator} yields
		$\eta_j=\eta_{j-1}+\zeta_j$. By conditional unbiasedness, independence within the fresh batch, and \eqref{eq:paired-smoothness}, we have
		$
		\E[\zeta_j\mid\mathcal H_j]=0,~
		\E[\norm{\zeta_j}^2\mid\mathcal H_j]
		\le\frac{\ell^2}{b}\norm{u_j-u_{j-1}}^2.
		$
		Since $\eta_{j-1}$ is $\mathcal H_j$-measurable, we obtain
		\begin{align}
			\E[\norm{\eta_j}^2\mid\mathcal H_j]
			&=\norm{\eta_{j-1}}^2+\E[\norm{\zeta_j}^2\mid\mathcal H_j]\le\norm{\eta_{j-1}}^2+\frac{\ell^2}{b}\norm{u_j-u_{j-1}}^2.
			\label{eq:vr-one-query-error}
		\end{align}
		Let $r(j):=q\lfloor j/q\rfloor$ be the last refresh index.
		Taking total expectations in~\eqref{eq:vr-one-query-error} and
		iterating from the refresh estimate at $r(j)$ gives
		\begin{align}\label{eq:vr-block-error}
			\E\norm{\eta_j}^2
			\le\frac{\sigma^2}{B_{\rm r}}
			+\frac{\ell^2}{b}\sum_{i=r(j)+1}^{j}\E\norm{u_i-u_{i-1}}^2.
		\end{align}
		Summing~\eqref{eq:vr-block-error} over $0\le j<3T$, each motion term is counted at most $q$ times. Since $u_{3t}=u_{3t-1}$ for $t\ge1$,~\eqref{eq:vr-motion} gives
		\begin{align*}
			\mathcal E_T
			&\le\frac{3T\sigma^2}{B_{\rm r}}
			+\frac{q\ell^2}{b}\sum_{j=1}^{3T-1}\E\norm{u_j-u_{j-1}}^2\le\frac{3T\sigma^2}{B_{\rm r}}
			+\frac{q\ell^2}{b}\left[
			\frac{\mathcal A_T}{L\beta}+\frac{\mathcal E_T}{L^2}\right],
		\end{align*}
		where the last step also uses~\eqref{eq:vr-cumulative-definitions}.
		Substituting $\chi=\ell/L$ then proves~\eqref{eq:vr-cumulative-error}.
	\end{proof}

	\begin{lemma}\label{lem:vr-budget}
		Suppose that Assumptions~\ref{ass:model},~\ref{ass:oracle}, and
		\ref{ass:paired} hold, and let the iterates be generated by
		Algorithm~\ref{alg:vr} with the parameter choices in~\eqref{eq:step-damping}
		and~\eqref{eq:center-relaxation}. Suppose further that
		\begin{align}\label{eq:vr-small-batch-condition}
			b\ge\frac{16C_\star q\chi^2}{\beta^2},
			\qquad
			M_T^{\rm vr}:=B+\frac{6C_\star T\sigma^2}{L\beta B_{\rm r}}.
		\end{align}
		Then Algorithm~\ref{alg:vr} satisfies
		\begin{align}\label{eq:vr-retained-budget}
			\frac18\mathcal A_T
			+\frac\beta{16L}\sum_{t=0}^{T-1}\E\norm{\nabla p(z_t)}^2
			\le M_T^{\rm vr},
		\end{align}
		and its output obeys
		\begin{align}\label{eq:vr-residual-bound}
			\E\mathcal R(x_{J+1},y_{J+1})^2
			\le\frac{128LB}{\beta T}
			+\frac{768C_\star\sigma^2}{\beta^2B_{\rm r}}
			+2\tau^2D_{\mathcal Y}^2.
		\end{align}
	\end{lemma}
	\begin{proof}
		The batch condition~\eqref{eq:vr-small-batch-condition}
		implies $q\chi^2/b\le\beta^2/(16C_\star)\le1/2$.
		Therefore, moving the last term in~\eqref{eq:vr-cumulative-error} to the left-hand side yields
		\begin{align}\label{eq:vr-error-absorb}
			\mathcal E_T\le\frac{6T\sigma^2}{B_{\rm r}}
			+\frac{2q\chi^2L}{b\beta}\mathcal A_T.
		\end{align}
		Taking expectations and summing~\eqref{eq:vr-pathwise}, using
		$\mathcal V_T\ge0$ and $\mathcal V_0\le B$ from
		Lemmas~\ref{lem:positive} and~\ref{lem:initial}, respectively,
		and substituting~\eqref{eq:vr-error-absorb}, we obtain
		\begin{align*}
			\frac14\mathcal A_T
			+\frac\beta{16L}\sum_{t<T}\E\norm{\nabla p(z_t)}^2
			&\le B+\frac{C_\star}{L\beta}\mathcal E_T
			\le B+\frac{6C_\star T\sigma^2}{L\beta B_{\rm r}}
			+\frac{2C_\star q\chi^2}{b\beta^2}\mathcal A_T
			\le M_T^{\rm vr}+\frac18\mathcal A_T,
		\end{align*}
		which proves~\eqref{eq:vr-retained-budget}.
		The terms of $\mathcal D_t^0$ in~\eqref{fc:dissipation} and
		\eqref{eq:vr-retained-budget} give
		\begin{align}\label{eq:vr-component-budgets}
				\sum_{t<T}\E\norm{P_t^+}^2&\le\frac{32M_T^{\rm vr}}h,\qquad
				\sum_{t<T}\E\norm{Q_t^+}^2\le\frac{LM_T^{\rm vr}}{3\beta},\qquad
				\sum_{t<T}\E\norm{v_{t+1}^0}^2\le\frac{8LM_T^{\rm vr}}{7\beta},\nonumber\\
				\sum_{t<T}\E\norm{X_t^+}^2&\le\frac{M_T^{\rm vr}}{7L\beta},\qquad
				\sum_{t<T}\E\norm{\nabla p(z_t)}^2\le\frac{16LM_T^{\rm vr}}\beta.
		\end{align}
		For $S_t$ defined in Lemma \ref{lem:budget}, the identities established in the proof of that lemma yield
		$
		S_t\le3\norm{P_t^+}^2+12L^2\norm{X_t^+}^2
		+3\norm{\nabla p(z_t)}^2+2\norm{Q_t^+}^2+2\norm{v_{t+1}^0}^2.
		$
		Substitute~\eqref{eq:vr-component-budgets} into the expected sum
		of this inequality, we have
		\begin{align}
			\sum_{t<T}\E S_t
			&\le\left(\frac{96\beta}{hL}+\frac{12}7+48+\frac23+\frac{16}7\right)
			\frac{LM_T^{\rm vr}}\beta
			\le\frac{64LM_T^{\rm vr}}\beta.
			\label{eq:vr-certificate-sum}
		\end{align}
		Since
		$\mathcal R(x_{t+1},y_{t+1})^2\le2S_t+2\tau^2D_{\mathcal Y}^2$,
		averaging with the independent uniform index $J$ now gives
		$
		\E\mathcal R(x_{J+1},y_{J+1})^2
		\le\frac{128LM_T^{\rm vr}}{\beta T}+2\tau^2D_{\mathcal Y}^2.
		$
		Substituting the definition of $M_T^{\rm vr}$ then proves
		\eqref{eq:vr-residual-bound}.
	\end{proof}
	
	\subsection{Complexity results}
	
	\subsubsection{Nonconvex--strongly concave setting}\label{sec:vr-ncsc}
	
	In this subsection, we again set $\tau=0$ in the definitions of $G$ and $p$. VR-SPDE applies Algorithm~\ref{alg:vr} to the query stream generated by SPDE, with all explicit dual-regularization terms removed and the parameters specified in~\eqref{eq:sc-parameters}.
	
	\begin{lemma}
		\label{lem:sc-vr-residual}
		Suppose that Assumptions~\ref{ass:model},~\ref{ass:oracle},~\ref{ass:strong-concavity}, and~\ref{ass:paired} hold and that
		the batch-size condition~\eqref{eq:vr-small-batch-condition}
		holds with $\beta=\beta_{\rm sc}$.
		Then VR-SPDE satisfies
		\begin{align}\label{eq:sc-vr-residual}
			\E\mathcal R(x_{J+1},y_{J+1})^2
			\le \frac{128LB}{\beta_{\rm sc}T}
			+\frac{768C_\star\sigma^2}{\beta_{\rm sc}^2B_{\rm r}}.
		\end{align}
	\end{lemma}
	\begin{proof}
		Use $\mathcal V^{\rm sc}$  introduced in the proof of
		Lemma~\ref{lem:sc-plain-residual}, with $G,p$ evaluated at $\tau=0$,
		$s=s_{\rm sc}$, and the dual-distance term given by 
		$\mu\norm{y-y^\star(z)}^2/256$.
		In the curvature and dissipation estimates, replace the dual strong-concavity modulus $\tau$ by $\mu$ and set
		$(h,\alpha,k,\beta)=(h_{\rm sc},\alpha_{\rm sc},k_{\rm sc},\beta_{\rm sc})$
		as specified in~\eqref{eq:sc-parameters}.
		The proofs of Lemmas~\ref{lem:vr-pathwise}--\ref{lem:vr-budget} then carry over directly. 
		Since $\tau=0$, the regularization-bias term vanishes,
		yielding~\eqref{eq:sc-vr-residual}.
	\end{proof}
	
	\begin{theorem}
		\label{thm:sc-vr-complexity}
		Suppose that Assumptions~\ref{ass:model},~\ref{ass:oracle},
		\ref{ass:paired}, and~\ref{ass:strong-concavity} hold. Fix a known bound
		$\overline B\ge B$ and set
		\begin{align}\label{eq:sc-vr-parameters}
				T&=\max\left\{1,\left\lceil
				\frac{256L\overline B}{\beta_{\rm sc}\eps^2}\right\rceil\right\},
				\qquad
				B_{\rm r}=\max\left\{1,\left\lceil
				\frac{1536C_\star\sigma^2}{\beta_{\rm sc}^2\eps^2}\right\rceil\right\},\nonumber\\\
				A_{{\rm vr},{\rm sc}}&=\frac{16C_\star\chi^2}{\beta_{\rm sc}^2},
				\qquad
				q=\max\left\{1,\left\lfloor
				\sqrt{\frac{B_{\rm r}}{A_{{\rm vr},{\rm sc}}}}\right\rfloor\right\},
				\qquad b=\lceil A_{{\rm vr},{\rm sc}}q\rceil.
		\end{align}
		Then VR-SPDE returns a point satisfying~\eqref{eq:goal}, and
		\begin{align}\label{eq:sc-vr-oracle}
			N_{{\rm vr},{\rm sc}}
			\le B_{\rm r}+18T\min\left\{B_{\rm r},
			\sqrt{A_{{\rm vr},{\rm sc}}B_{\rm r}}\right\}=O\!\left(\kappa\eps^{-2}
			+\chi\kappa^{3/2}\eps^{-3}\right).
		\end{align}
		the iteration count and  batch size are
		$
		T=O(\sqrt\kappa\eps^{-2}),~
		B_{\rm r}=O(\kappa\eps^{-2}),~
		q=\Theta((\chi\eps)^{-1}),~
		b=O(\chi\kappa\eps^{-1}).
		$
	\end{theorem}
	\begin{proof}
		By~\eqref{eq:sc-vr-parameters}, we have $b\ge A_{{\rm vr},{\rm sc}}q$, so the batch-size condition in Lemma~\ref{lem:sc-vr-residual} holds. The choices of $T$ and
		$B_{\rm r}$ ensure that
		$
		\frac{128LB}{\beta_{\rm sc}T}\le\frac{\eps^2}{2},
		~
		\frac{768C_\star\sigma^2}{\beta_{\rm sc}^2B_{\rm r}}
		\le\frac{\eps^2}{2}.
		$
		Substituting these bounds into~\eqref{eq:sc-vr-residual} proves~\eqref{eq:goal}.
		Among the $3T$ oracle requests, $\lceil3T/q\rceil$ are refresh steps. Each refresh requires $B_{\rm r}$ oracle calls, while each nonrefresh request requires at most $2b$. If $q\ge2$,~\eqref{eq:sc-vr-parameters} further gives
		\[
		\frac12\sqrt{\frac{B_{\rm r}}{A_{{\rm vr},{\rm sc}}}}\le q,
		\qquad
		b\le2\sqrt{A_{{\rm vr},{\rm sc}}B_{\rm r}}.
		\]
		The request count gives
		$
		N_{{\rm vr},{\rm sc}}
		\le B_{\rm r}+\frac{3TB_{\rm r}}q+6Tb
		\le B_{\rm r}+18T\sqrt{A_{{\rm vr},{\rm sc}}B_{\rm r}}.
		$
		Since $B_{\rm r}\ge4A_{{\rm vr},{\rm sc}}$, this proves
		the bound in~\eqref{eq:sc-vr-oracle} for $q\ge2$.
		If $q=1$, then $B_{\rm r}<4A_{{\rm vr},{\rm sc}}$ and
		$
		N_{{\rm vr},{\rm sc}}=3TB_{\rm r}
		\le6T\min\{B_{\rm r},\sqrt{A_{{\rm vr},{\rm sc}}B_{\rm r}}\},
		$
		so the same bound holds.
		For fixed $L,\overline B$ and positive $\sigma,\chi$,
		using $\beta_{\rm sc}=\Theta(\kappa^{-1/2})$ in~\eqref{eq:sc-vr-parameters} gives
		\begin{align*}
			T&=O(\sqrt\kappa\eps^{-2}),\quad
			B_{\rm r}=\Theta(\kappa\eps^{-2}),\quad
			A_{{\rm vr},{\rm sc}}=\Theta(\chi^2\kappa),\quad
			q=\Theta((\chi\eps)^{-1}),\quad
			b=O(\chi\kappa\eps^{-1}).
		\end{align*}
		Substituting these estimates into~\eqref{eq:sc-vr-oracle}
		gives the stated SFO complexity.
	\end{proof}
	
	\begin{remark}
		Theorem~\ref{thm:sc-vr-complexity} shows that VR-SPDE finds an
		$\eps$-game-stationary point for stochastic nonconvex--strongly concave
		problems using $O(\kappa^{3/2}\eps^{-3})$ SFO calls.
	\end{remark}
	
	\subsubsection{Nonconvex--concave setting}\label{sec:vr-ncc}
	
	\begin{theorem}\label{thm:vr-complexity}
		Suppose that Assumptions~\ref{ass:model},~\ref{ass:oracle}, and
		\ref{ass:paired} hold. Fix a known bound $\overline B\ge B$ independent of $\eps,\tau$.
		For $\eps>0$, choose $\tau$ as in~\eqref{eq:tau-choice}, choose
		$h,\alpha,k$ by~\eqref{eq:step-damping}, and choose $\beta$
		by~\eqref{eq:center-relaxation}. Set
		\begin{align}\label{eq:vr-parameters}
				T&=\max\left\{1,\left\lceil\frac{512L\overline B}{\beta\eps^2}\right\rceil\right\},\qquad
				B_{\rm r}=\max\left\{1,\left\lceil\frac{3072C_\star\sigma^2}{\beta^2\eps^2}\right\rceil\right\},\nonumber\\
				A_{\rm vr}&=\frac{16C_\star\chi^2}{\beta^2},\qquad
				q=\max\left\{1,\left\lfloor\sqrt{\frac{B_{\rm r}}{A_{\rm vr}}}\right\rfloor\right\},\qquad
				b=\lceil A_{\rm vr}q\rceil.
		\end{align}
		Then Algorithm~\ref{alg:vr} satisfies~\eqref{eq:goal}.
		Its number of stochastic full-gradient oracle calls obeys
		\begin{align}
			N_{\rm vr}
			&\le B_{\rm r}+18T\min\{B_{\rm r},\sqrt{A_{\rm vr}B_{\rm r}}\}=O(\eps^{-9/2}).
			\label{eq:vr-oracle-budget}
		\end{align}
		the iteration count and  batch size are
		$
		T=O(\eps^{-5/2}),~B_{\rm r}=O(\eps^{-3}),~
		q=\Theta(\eps^{-1}),~b=O(\eps^{-2}).
		$
	\end{theorem}
	\begin{proof}
		The choice $b=\lceil A_{\rm vr}q\rceil$ ensures that the batch-size condition in Lemma~\ref{lem:vr-budget} is satisfied. Moreover, the choices of $T$ and $B_{\rm r}$ in~\eqref{eq:vr-parameters}, together with~\eqref{eq:tau-choice}, imply
		\begin{align}\label{eq:vr-error-allocation}
			\frac{128LB}{\beta T}\le\frac{\eps^2}{4},\qquad
			\frac{768C_\star\sigma^2}{\beta^2B_{\rm r}}\le\frac{\eps^2}{4},\qquad
			2\tau^2D_{\mathcal Y}^2\le\frac{\eps^2}{2}.
		\end{align}
		Substituting these bounds into~\eqref{eq:vr-residual-bound} gives
		$\E\mathcal R^2\le\eps^2$.
		The $3T$ requests contain $\lceil3T/q\rceil$ refreshes.
		Each refresh costs $B_{\rm r}$ oracle calls and each remaining
		request costs at most $2b$, so
		\[
		N_{\rm vr}\le B_{\rm r}\left\lceil\frac{3T}{q}\right\rceil
		+2b\left(3T-\left\lceil\frac{3T}{q}\right\rceil\right).
		\]
		If $q\ge2$, the definitions of $q$ and $b$ in~\eqref{eq:vr-parameters} yield
		$
		\frac12\sqrt{\frac{B_{\rm r}}{A_{\rm vr}}}\le q
		\le\sqrt{\frac{B_{\rm r}}{A_{\rm vr}}},~
		b\le\sqrt{A_{\rm vr}B_{\rm r}}+1
		\le2\sqrt{A_{\rm vr}B_{\rm r}}.
		$
		Substituting these bounds into the request count gives
		\[
		N_{\rm vr}\le B_{\rm r}+\frac{3TB_{\rm r}}q+6Tb
		\le B_{\rm r}+18T\sqrt{A_{\rm vr}B_{\rm r}}.
		\]
		Since $B_{\rm r}\ge4A_{\rm vr}$, this proves the bound
		in~\eqref{eq:vr-oracle-budget} for $q\ge2$.
		If $q=1$, then $B_{\rm r}<4A_{\rm vr}$ and
		$
		N_{\rm vr}=3TB_{\rm r}
		\le6T\min\{B_{\rm r},\sqrt{A_{\rm vr}B_{\rm r}}\},
		$
		so the same bound holds.
		For fixed problem data and $\overline B$, with $\sigma>0$,
		substituting $\beta=\Theta(\eps^{1/2})$ into~\eqref{eq:vr-parameters} gives
		\begin{align*}
			T&=O(\eps^{-5/2}),\quad
			B_{\rm r}=\Theta(\eps^{-3}),\quad A_{\rm vr}=\Theta(\eps^{-1}),\quad
			q=\Theta(\eps^{-1}),\quad b=O(\eps^{-2}).
		\end{align*}
		Substitution into~\eqref{eq:vr-oracle-budget}
		yields $N_{\rm vr}=O(\eps^{-9/2})$.
	\end{proof}
	
	\begin{remark}
		Theorem~\ref{thm:vr-complexity} shows that VR-SPDE finds an
		$\eps$-game-stationary point for stochastic nonconvex--concave
		problems using $O(\eps^{-4.5})$ SFO calls.
	\end{remark}
	
	\section{Optimization-stationarity guarantees}\label{sec:os}
	We now turn from game stationarity of a feasible primal--dual
	pair to the optimization-stationarity criterion
	in~\eqref{eq:os-criterion}.
	With the parameter choices specified below, SPDE and VR-SPDE
	provide this guarantee through their center sequences,
	without additional stochastic-gradient evaluations.
	
	For $0\le\tau\le L$, define the regularized value function
	and its $1/(2L)$-Moreau envelope over $\mathcal X$ by
	\begin{align}\label{eq:os-value-envelope}
		\phi_\tau(x):=\max_{y\in \mathcal Y}
		\left\{F(x,y)-\frac\tau2\norm y^2\right\},\qquad
		p_\tau(z):=\min_{x\in \mathcal X}
		\left\{\phi_\tau(x)+L\norm{x-z}^2\right\}.
	\end{align}
	Here $\phi_0=\phi$ and $p_0$ is the original envelope
	in~\eqref{eq:os-original-envelope}.
	The notation $p_\tau$ agrees with $p$ in the regularized analysis
	of Section~\ref{sec:plain}, with $\tau=0$ in the strongly concave
	setting.
	By Assumption~\ref{ass:model}, each $\phi_\tau$ is $L$-weakly convex,
	so the envelope minimizer is unique and
	\begin{align}\label{eq:os-envelope-gradient}
		x_\tau^\star(z)
		&:=\argmin_{x\in\mathcal X}
		\left\{\phi_\tau(x)+L\norm{x-z}^2\right\},\quad
		\nabla p_\tau(z)=2L\bigl(z-x_\tau^\star(z)\bigr).
	\end{align}
	For the OS guarantees, both algorithms additionally return
	$z_{\rm out}:=z_J$, using the same independent index
	$J\sim\operatorname{Unif}\{0,\ldots,T-1\}$ as before.
	
	%
	
	\begin{lemma}\label{lem:os-transfer}
		Suppose that Assumption~\ref{ass:model} holds.  For every
		$0\le\tau\le L$ and $z\in\mathbb R^n$, $x\in \mathcal X$, we have
		\begin{align*}
			0&\le\phi(x)-\phi_\tau(x)\le\frac\tau2D_{\mathcal Y}^2,\quad
			\norm{x_\tau^\star(z)-x_0^\star(z)}^2
			\le\frac{\tau D_{\mathcal Y}^2}{L},\quad
			\norm{\nabla p_0(z)-\nabla p_\tau(z)}
			\le2D_{\mathcal Y}\sqrt{L\tau}.
		\end{align*}
		In particular,
		\begin{align}\label{eq:os-squared-transfer}
			\mathcal S_{\rm OS}(z)^2
			\le2\norm{\nabla p_\tau(z)}^2+8L\tau D_{\mathcal Y}^2.
		\end{align}
	\end{lemma}
	\begin{proof}
		We first bound the bias in the value function. For any $x\in\mathcal X$,
		choose $y_0(x)\in\operatorname*{argmax}_{y\in\mathcal Y}F(x,y)$.
		By the definition of $\phi_\tau$ and the bound $\norm{y_0(x)}\le D_{\mathcal Y}$, we have
		\[
		\phi(x)-\frac\tau2D_{\mathcal Y}^2
		\le F(x,y_0(x))-\frac\tau2\norm{y_0(x)}^2
		\le\phi_\tau(x)\le\phi(x).
		\]
		For fixed $z$, let
		$
		P_{\tau,z}(x):=\phi_\tau(x)+L\norm{x-z}^2,
		~ \delta_\tau:=\frac\tau2D_{\mathcal Y}^2.
		$
		Since $\phi_\tau$ is $L$-weakly convex, $P_{\tau,z}$ is $L$-strongly convex.
		By strong convexity, the optimality of $x_0^\star(z)$ and $x_\tau^\star(z)$, we get
		\begin{align}
			\frac L2\norm{x_\tau^\star(z)-x_0^\star(z)}^2
			&\le P_{0,z}(x_\tau^\star(z))-P_{0,z}(x_0^\star(z))
			\le P_{\tau,z}(x_0^\star(z))+\delta_\tau
			-P_{0,z}(x_0^\star(z))\le\delta_\tau.\label{eq:os-prox-transfer-proof}
		\end{align}
		By~\eqref{eq:os-prox-transfer-proof} and~\eqref{eq:os-envelope-gradient}, we have
		$
		\norm{\nabla p_0(z)-\nabla p_\tau(z)}
		=2L\norm{x_\tau^\star(z)-x_0^\star(z)}
		\le2D_{\mathcal Y}\sqrt{L\tau},
		$
		Using this bound,~\eqref{eq:os-criterion}, and
		$\norm{a+b}^2\le2\norm a^2+2\norm b^2$, we obtain
		\begin{align*}
			\mathcal S_{\rm OS}(z)^2
			&=\norm{\nabla p_\tau(z)+[\nabla p_0(z)-\nabla p_\tau(z)]}^2
			\le2\norm{\nabla p_\tau(z)}^2+8L\tau D_{\mathcal Y}^2,
		\end{align*}
		proving~\eqref{eq:os-squared-transfer}.
	\end{proof}
	
	\begin{lemma}\label{lem:os-envelope-budgets}
		Suppose that Assumptions~\ref{ass:model} and~\ref{ass:oracle} hold, and let
		$J$ be independent and uniform on $\{0,\ldots,T-1\}$.  For SPDE with
		$0<\tau\le L$ and the parameters in~\eqref{eq:step-damping} and
		\eqref{eq:center-relaxation}, we have
		\begin{align}\label{eq:os-spde-envelope-budget}
			\E\norm{\nabla p_\tau(z_J)}^2
			\le\frac{8LB}{\beta T}+\frac{2400\sigma^2}{\beta m}.
		\end{align}
		If Assumption~\ref{ass:paired} also holds and
		$b\ge16C_\star q\chi^2/\beta^2$, then VR-SPDE satisfies
		\begin{align}\label{eq:os-vr-envelope-budget}
			\E\norm{\nabla p_\tau(z_J)}^2
			\le\frac{16LB}{\beta T}
			+\frac{96C_\star\sigma^2}{\beta^2B_{\rm r}}.
		\end{align}
		Under Assumption~\ref{ass:strong-concavity},
		SPDE and VR-SPDE with $\tau=0$ satisfy
		\eqref{eq:os-spde-envelope-budget} and
		\eqref{eq:os-vr-envelope-budget}, respectively,
		with $(p_\tau,\beta)$ replaced by $(p_0,\beta_{\rm sc})$.
	\end{lemma}
	\begin{proof}
		By the independence and uniformity of $J$, we have
		$
		\E\norm{\nabla p_\tau(z_J)}^2
		=\frac1T\sum_{t=0}^{T-1}
		\E\norm{\nabla p_\tau(z_t)}^2.
		$
		Using the last bounds in~\eqref{eq:component-budgets}
		and~\eqref{eq:vr-component-budgets}, together with~\eqref{eq:vr-small-batch-condition},
		we obtain, for SPDE and VR-SPDE respectively,
		\begin{align*}
			\E\norm{\nabla p_\tau(z_J)}^2
			&\le\frac{8L}{\beta T}
			\left(B+\frac{300T\sigma^2}{Lm}\right)
			=\frac{8LB}{\beta T}+\frac{2400\sigma^2}{\beta m},\\
			\E\norm{\nabla p_\tau(z_J)}^2
			&\le\frac{16L}{\beta T}
			\left(B+\frac{6C_\star T\sigma^2}
			{L\beta B_{\rm r}}\right)=\frac{16LB}{\beta T}
			+\frac{96C_\star\sigma^2}{\beta^2B_{\rm r}}.
		\end{align*}
		For SPDE and VR-SPDE with $\tau=0$, apply the same averaging identity
		to the cumulative bounds in the proofs of
		Lemmas~\ref{lem:sc-plain-residual} and~\ref{lem:sc-vr-residual},
		with $(p_\tau,\beta)$ replaced by $(p_0,\beta_{\rm sc})$.
	\end{proof}
	
	\subsection{Nonconvex--strongly concave setting}
	
	\begin{theorem}\label{thm:os-ncsc-complexity}
		Suppose that Assumptions~\ref{ass:model},~\ref{ass:oracle}, and
		\ref{ass:strong-concavity} hold.  Let $\eps>0$ and fix a known bound
		$\overline B\ge B$.
		For SPDE, use
		\eqref{eq:sc-parameters}, and set
		\begin{align}\label{eq:os-ncsc-spde-budgets}
			T=\max\left\{1,\left\lceil
			\frac{16L\overline B}{\beta_{\rm sc}\eps^2}\right\rceil\right\},
			\qquad
			m=\max\left\{1,\left\lceil
			\frac{4800\sigma^2}{\beta_{\rm sc}\eps^2}\right\rceil\right\}.
		\end{align}
		Then $\E\mathcal S_{\rm OS}(z_{\rm out})^2\le\eps^2$, and the number of SFO calls
		satisfies
		\begin{align}\label{eq:os-ncsc-spde-complexity}
			N_{{\rm os},{\rm sc}}
			\le3\left(1+\frac{16L\overline B}
			{\beta_{\rm sc}\eps^2}\right)
			\left(1+\frac{4800\sigma^2}
			{\beta_{\rm sc}\eps^2}\right)
			=O(\kappa\eps^{-4}).
		\end{align}
		
		If Assumption~\ref{ass:paired} also holds, run VR-SPDE
		specialization with
		\begin{align}\label{eq:os-ncsc-vr-budgets}
				T&=\max\left\{1,\left\lceil
				\frac{32L\overline B}{\beta_{\rm sc}\eps^2}\right\rceil\right\},
				\qquad
				B_{\rm r}=\max\left\{1,\left\lceil
				\frac{192C_\star\sigma^2}
				{\beta_{\rm sc}^2\eps^2}\right\rceil\right\},\nonumber\\
				A_{{\rm os},{\rm sc}}&=\frac{16C_\star\chi^2}{\beta_{\rm sc}^2},
				\qquad
				q=\max\left\{1,\left\lfloor
				\sqrt{\frac{B_{\rm r}}{A_{{\rm os},{\rm sc}}}}\right\rfloor\right\},
				\qquad b=\left\lceil A_{{\rm os},{\rm sc}}q\right\rceil.
		\end{align}
		Then $\E\mathcal S_{\rm OS}(z_{\rm out})^2\le\eps^2$ and the number of SFO calls	satisfies
		\begin{align}
			N_{{\rm os},{\rm vr},{\rm sc}}
			&\le B_{\rm r}+18T\min\left\{B_{\rm r},
			\sqrt{A_{{\rm os},{\rm sc}}B_{\rm r}}\right\}=O\!\left(\kappa\eps^{-2}
			+\chi\kappa^{3/2}\eps^{-3}\right).
			\label{eq:os-ncsc-vr-complexity}
		\end{align}
	\end{theorem}
	\begin{proof}
		For SPDE with $\tau=0$, Lemma~\ref{lem:os-envelope-budgets}
		and~\eqref{eq:os-ncsc-spde-budgets} give
		\begin{align*}
			\E\mathcal S_{\rm OS}(z_{\rm out})^2
			&=\E\norm{\nabla p_0(z_J)}^2
			\le\frac{8LB}{\beta_{\rm sc}T}
			+\frac{2400\sigma^2}{\beta_{\rm sc}m}\le\eps^2.
		\end{align*}
		Using $N_{{\rm os},{\rm sc}}=3Tm$,
		$\max\{1,\lceil a\rceil\}\le1+a$ for $a\ge0$, and
		$\beta_{\rm sc}=\Theta(\kappa^{-1/2})$ yields~\eqref{eq:os-ncsc-spde-complexity}.
		For VR-SPDE, the choice $b=\lceil A_{{\rm os},{\rm sc}}q\rceil$
		satisfies~\eqref{eq:vr-small-batch-condition} with $\beta=\beta_{\rm sc}$.
		Thus Lemma~\ref{lem:os-envelope-budgets} and~\eqref{eq:os-ncsc-vr-budgets} yield
		\begin{align*}
			\E\mathcal S_{\rm OS}(z_{\rm out})^2
			&\le\frac{16LB}{\beta_{\rm sc}T}
			+\frac{96C_\star\sigma^2}
			{\beta_{\rm sc}^2B_{\rm r}}\le\eps^2.
		\end{align*}
		The counting argument for~\eqref{eq:sc-vr-oracle}, with
		$A_{{\rm vr},{\rm sc}}$ replaced by $A_{{\rm os},{\rm sc}}$, gives the
		first bound in~\eqref{eq:os-ncsc-vr-complexity}.
		For fixed $L,\overline B$ and positive $\sigma,\chi$,
		using $\beta_{\rm sc}=\Theta(\kappa^{-1/2})$
		in~\eqref{eq:os-ncsc-vr-budgets} gives
		$
		T=O(\sqrt\kappa\eps^{-2}),~
		B_{\rm r}=O(\kappa\eps^{-2}),~
		A_{{\rm os},{\rm sc}}=O(\chi^2\kappa).
		$
		Substitution into~\eqref{eq:os-ncsc-vr-complexity} gives the stated SFO order.
	\end{proof}
	
	\subsection{Nonconvex--concave setting}
	
	\begin{theorem}\label{thm:os-ncc-complexity}
		Suppose that Assumptions~\ref{ass:model} and~\ref{ass:oracle} hold.  Let
		$\eps>0$, fix a known bound $\overline B\ge B$, and set
		\begin{align}\label{eq:os-tau-choice}
			\tau_{\rm os}:=\min\left\{L,
			\frac{\eps^2}{16LD_{\mathcal Y}^2}\right\},
			\qquad
			K_\eps^{\rm os}:=\max\left\{1,\frac{4LD_{\mathcal Y}}{\eps}\right\}.
		\end{align}
		Choose $h,\alpha,k$ by~\eqref{eq:step-damping} and $\beta$ by
		\eqref{eq:center-relaxation}, with $\tau=\tau_{\rm os}$.		
		For SPDE, set
		\begin{align}\label{eq:os-ncc-spde-budgets}
			T=\max\left\{1,\left\lceil
			\frac{64L\overline B}{\beta\eps^2}\right\rceil\right\},
			\qquad
			m=\max\left\{1,\left\lceil
			\frac{19200\sigma^2}{\beta\eps^2}\right\rceil\right\}.
		\end{align}
		Then $\E\mathcal S_{\rm OS}(z_{\rm out})^2\le\eps^2$ and  the number of SFO calls	satisfies
		\begin{align}
			N_{{\rm os},{\rm ncc}}
			&\le3\left(1+\frac{64L\overline B}{\beta\eps^2}\right)
			\left(1+\frac{19200\sigma^2}{\beta\eps^2}\right)=O\!\left(1+
			\frac{(L\overline B+\sigma^2)K_\eps^{\rm os}}{\eps^2}
			+\frac{L\overline B\sigma^2(K_\eps^{\rm os})^2}{\eps^4}
			\right).
			\label{eq:os-ncc-spde-complexity}
		\end{align}
		For fixed problem data, $\overline B$, and $\sigma>0$, this gives
		$T=m=O(\eps^{-3})$ and $N_{{\rm os},{\rm ncc}}=O(\eps^{-6})$.
		
		If Assumption~\ref{ass:paired} also holds, run VR-SPDE with
		\begin{align}\label{eq:os-ncc-vr-budgets}
				T&=\max\left\{1,\left\lceil
				\frac{128L\overline B}{\beta\eps^2}\right\rceil\right\},
				\qquad
				B_{\rm r}=\max\left\{1,\left\lceil
				\frac{768C_\star\sigma^2}{\beta^2\eps^2}\right\rceil\right\},\nonumber\\
				A_{{\rm os},{\rm ncc}}&=\frac{16C_\star\chi^2}{\beta^2},
				\qquad
				q=\max\left\{1,\left\lfloor
				\sqrt{\frac{B_{\rm r}}{A_{{\rm os},{\rm ncc}}}}\right\rfloor\right\},
				\qquad b=\left\lceil A_{{\rm os},{\rm ncc}}q\right\rceil.
		\end{align}
		Then $\E\mathcal S_{\rm OS}(z_{\rm out})^2\le\eps^2$ and the number of SFO calls	satisfies
		\begin{align}
			N_{{\rm os},{\rm vr},{\rm ncc}}
			&\le B_{\rm r}+18T\min\left\{B_{\rm r},
			\sqrt{A_{{\rm os},{\rm ncc}}B_{\rm r}}\right\}=O(\eps^{-4}+\chi\eps^{-6})
			=O(\eps^{-6}).
			\label{eq:os-ncc-vr-complexity}
		\end{align}
	\end{theorem}
	\begin{proof}
		By~\eqref{eq:os-tau-choice}, we have
		$
		8L\tau_{\rm os}D_{\mathcal Y}^2\le\frac{\eps^2}{2}.
		$
		The VR batch size satisfies $b=\lceil A_{{\rm os},{\rm ncc}}q\rceil
		\ge16C_\star q\chi^2/\beta^2$.
		Substituting~\eqref{eq:os-ncc-spde-budgets} and~\eqref{eq:os-ncc-vr-budgets}
		into Lemma~\ref{lem:os-envelope-budgets} gives, for SPDE and VR-SPDE respectively,
		\begin{align}
			\E\norm{\nabla p_{\tau_{\rm os}}(z_J)}^2
			&\le\frac{8LB}{\beta T}
			+\frac{2400\sigma^2}{\beta m}\le\frac{\eps^2}{4},
			\quad
			\E\norm{\nabla p_{\tau_{\rm os}}(z_J)}^2
			\le\frac{16LB}{\beta T}
			+\frac{96C_\star\sigma^2}{\beta^2B_{\rm r}}\le\frac{\eps^2}{4}.
			\label{eq:os-ncc-vr-regularized}
		\end{align}
		Combining these bounds with~\eqref{eq:os-squared-transfer} gives, for either method,
		\[
		\E\mathcal S_{\rm OS}(z_{\rm out})^2
		\le2\E\norm{\nabla p_{\tau_{\rm os}}(z_J)}^2
		+8L\tau_{\rm os}D_{\mathcal Y}^2\le\eps^2.
		\]
		By~\eqref{eq:step-damping},~\eqref{eq:center-relaxation}, and~\eqref{eq:os-tau-choice}, we get
		$
		\frac1\beta=\Theta\!\left(\sqrt{\frac L{\tau_{\rm os}}}\right)
		=\Theta(K_\eps^{\rm os}).
		$
		For SPDE, using $N_{{\rm os},{\rm ncc}}=3Tm$,
		$\max\{1,\lceil a\rceil\}\le1+a$ for $a\ge0$,
		yields~\eqref{eq:os-ncc-spde-complexity}.
		This gives
		$T=m=O(\eps^{-3})$ and $N_{{\rm os},{\rm ncc}}=O(\eps^{-6})$.
		
		For VR-SPDE, applying the counting argument used to derive~\eqref{eq:vr-oracle-budget}, with $A_{\rm vr}$ replaced by $A_{{\rm os},{\rm ncc}}$, yields the first bound in~\eqref{eq:os-ncc-vr-complexity}. Moreover,~\eqref{eq:os-ncc-vr-budgets} gives
		$
		T=O(\eps^{-3}),~
		B_{\rm r}=O(\eps^{-4}),~
		A_{{\rm os},{\rm ncc}}=O(\chi^2\eps^{-2}).
		$
		Substituting these estimates into~\eqref{eq:os-ncc-vr-complexity} yields the stated SFO complexity.
	\end{proof}
	
	\section{Numerical experiments}\label{sec:numerical}
	We evaluate SPDE and VR-SPDE against representative stochastic
	first-order methods on two problems: a toy WGAN with a linear
	generator and a group-robust learning
	problem~\cite{hu2018distributionally,sagawa2020distributionally,yang2022faster}.
	
	\subsection{A Toy WGAN with linear generator}\label{subsec:numerical-toy-wgan}
	We consider the Toy WGAN with linear generator used in~\cite{yang2022faster},
	\begin{align}\label{eq:numerical-toy-wgan}
		\min _{x} \max _{y} ~\mathbb{E}_{\left(x^{real}, z\right) \sim \mathcal{D}}\left[D_y\left(x^{real}\right)-D_y\left(G_{x}(z)\right)\right]-\lambda\|y\|^2,
	\end{align}
	where \(x=(\varphi_1, \varphi_2)\in[-2,2]^2\),
	\(y=(\phi_1,\phi_2)\in\{v\in\R^2:\|v\|\le 5000\}\), $G_x(z)=\varphi_1+\varphi_2z$, $D_y(u)=\phi_1u+\phi_2u^2$, and $\lambda=10^{-3}$.
    We use $z\sim\mathcal N(0,1)$ and the coupled real sample $x^{\rm real}=0.1z$, whose distribution has mean zero and variance $0.01$.
    The saddle points are $(\varphi_1^*,\varphi_2^*,\phi_1^*,\phi_2^*)=(0,\pm0.1,0,0)$. Its generator solution set is
	\(\mathcal X^\star=\{(0,0.1),(0,-0.1)\}\), and we measure
	\(\operatorname{dist}(x,\mathcal X^\star)
	=\sqrt{\varphi_1^2+(|\varphi_2|-0.1)^2}\).

	We compare SGDA~\cite{lin2020gda}, Acc-MDA~\cite{huang2022accelerated}, SAPD+ and VR-SAPD+~\cite{zhang2022sapdplus}, SPDE, and VR-SPDE. For SGDA, we use
	\((\eta_x,\eta_y)=(0.05,0.10)\), obtained by jointly scaling the Toy WGAN values in~\cite{yang2022faster} by 0.10 to avoid immediate boundary saturation. The SPDE parameters are
	\((h,\beta,k,\alpha)=(0.24,0.05,0.14,0.20)\), and the VR-SPDE parameters are
	\((h,\beta,k,\alpha)=(0.26,0.09,0.02,0.15)\) and
	\((B_{\rm r},b,q)=(500,40,5)\). 
	All the parameters of the Acc-MDA, SAPD+ and VR-SAPD+ are chosen the same as that in \cite{huang2022accelerated} and~\cite{zhang2022sapdplus} respectively.
	\begin{figure}[!ht]
		\centering
		\includegraphics[width=\textwidth]{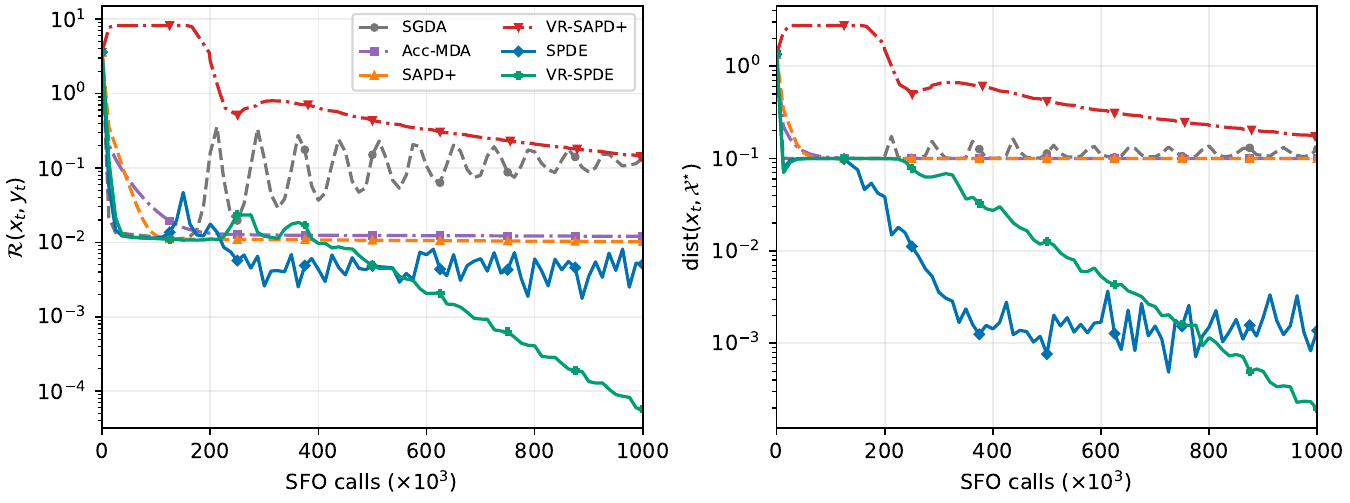}
		\caption{Performance of six algorithms for Toy WGAN with linear generator problem.}
		\label{fig:numerical-toy-wgan}
	\end{figure}
	
	Figure~\ref{fig:numerical-toy-wgan} shows the 
	game-stationarity residual $\mathcal R(x_t,y_t)$
	defined in~\eqref{eq:residual}
	and the generator distance $\operatorname{dist}(x_t,\mathcal X^\star)$
	against the number of SFO calls.
	At the end of the reported runs, both SPDE and VR-SPDE attain
	lower values of both metrics than the baselines, with VR-SPDE
	achieving the lowest values.

	\subsection{A group robust learning}\label{subsec:numerical-ncc}
	We next consider a group robust learning based on the LIBSVM w8a data set~\cite{chang2011libsvm}. Let ${(a_i,b_i)}{i=1}^{N}$ denote the data set, where $a_i\in\R^d$ and $b_i\in{-1,1}$. We partition the samples into the negative- and positive-label groups, with corresponding index sets $\mathcal I_1$ and $\mathcal I_2$, respectively. The group robust learning problem is
	\begin{align}\label{eq:numerical-ncc-dro}
		\min_{x\in\R^d}\max_{y\in\Delta_2}
		\sum_{g=1}^{2}y_gL_g(x)+r(x)
		=\min_{x\in\R^d}\bigl\{\max_{g\in\{1,2\}}L_g(x)+r(x)\bigr\},
	\end{align}
	where \(L_g(x)=|\mathcal I_g|^{-1}\sum_{i\in\mathcal I_g}\log(1+\exp(-b_i a_i^\top x))\), \(r(x)=\eta_1\sum_{j=1}^{d}\gamma x_j^2/(1+\gamma x_j^2)\), and  \(\Delta_2=\{y\in\R_+^2:y_1+y_2=1\}\).

	We compare SGDA, SAPD+, FORMDA~\cite{zhangxu2025formda}, SPDE, and VR-SPDE.  Set \((\gamma,\eta_1)=(10,10^{-3})\). Following the parameter structures in the cited papers, we tune one scalar parameter for each method over eight candidates using the same five pilot seeds and SFO budget. For SGDA, we use \((\eta_x,\eta_y)=(7^{-3},7^{-1})\). For SAPD+, the reference step sizes are multiplied by \(1.5\), giving \((\eta_x,\eta_y)=(0.15,0.015)\). For FORMDA, we use the parameter schedules specified in the original paper with \(\alpha_k=1.0/((k+12)^2/13+1)\). For SPDE, we use \((\tau,\alpha,k,\beta,h)=(0.05,0.9,0.1,0.5,0.15)\), while VR-SPDE uses \((B_{\rm r},b,q)=(32,2,20)\) and \((\tau,\alpha,k,\beta,h)=(0.05,0.8,0.05,0.5,0.30)\).

	\begin{figure}[!ht]
		\centering
		\includegraphics[width=\textwidth]{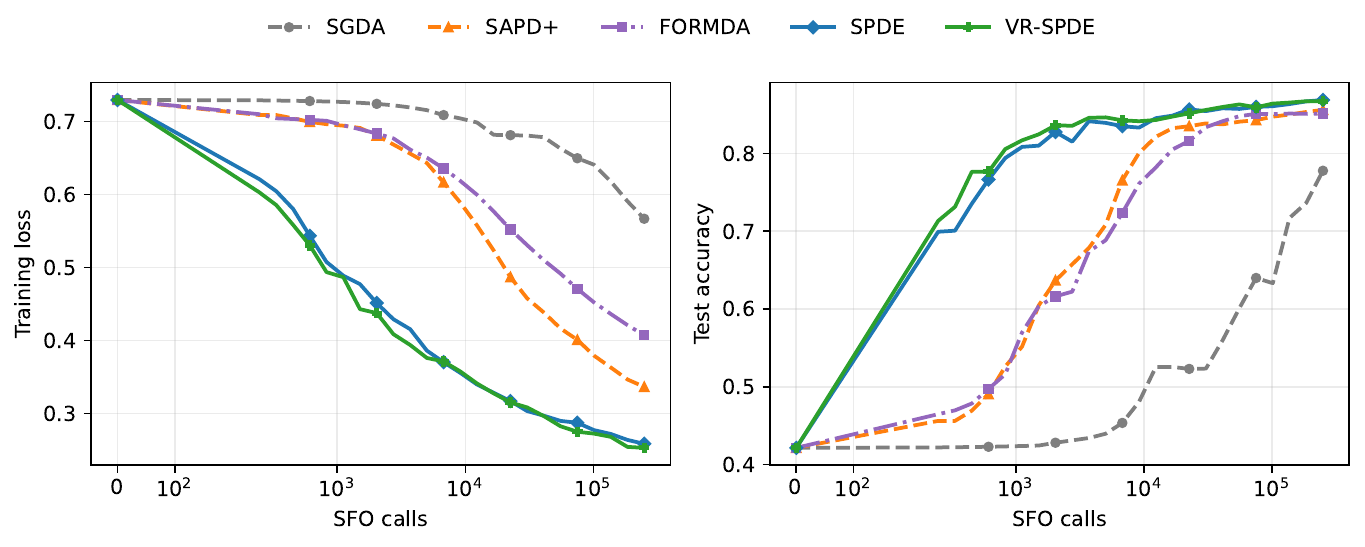}
		\caption{Performance of five algorithms for the group robust learning.}
		\label{fig:numerical-ncc}
	\end{figure}
	
	Figure~\ref{fig:numerical-ncc} shows that SPDE and VR-SPDE
	reduce both metrics faster than the baselines under the same
	SFO budget.

	\section{Conclusions}\label{sec:conclusion}
	This work shows that strong stationarity guarantees for stochastic
	nonconvex minimax optimization can be achieved without nested
	subproblem solvers.
	The key design couples projected predictor--corrector steps with
	damped dual momentum, normal-cone corrections, and a moving
	proximal center. Recursive variance reduction fits into the same
	single-loop structure.
	
	The resulting SPDE and VR-SPDE methods support both GS and OS
	guarantees with criterion-specific parameter choices.
	To the best of our knowledge, their total SFO bounds are the
	best known among single-loop stochastic first-order methods
	for the respective criteria and problem classes.
	Their OS bounds also match the best-known multi-loop dependence
	on target accuracy.
	Notably, variance reduction improves the GS bounds in both
	settings and the OS bound in the strongly concave setting,
	but yields no improvement in the stated OS bound for the merely
	concave setting. This distinction highlights the importance
	of analyzing the two stationarity criteria separately.

\end{document}